\documentclass[11pt,reqno]{article}

\usepackage[text={153mm,235mm},centering]{geometry}
\usepackage{enumerate}

\usepackage[mathscr]{eucal}
\usepackage[all]{xy}
\usepackage{mathrsfs}
\usepackage{authblk}
\usepackage{xypic}
\usepackage{amsfonts}
\usepackage{amsmath}
\usepackage{amsthm}
\usepackage{amssymb,bm}
\usepackage{tabularx}
\usepackage{graphicx}
\usepackage{wrapfig}
\usepackage{pict2e}
\usepackage{amsmath}
\usepackage{pifont}
\usepackage{hyperref}

\usepackage{pdfpages}

\usepackage{mathtools}

\usepackage{enumitem}

\allowdisplaybreaks
\newtheorem{theorem}{Theorem}[section]

\newtheorem{proposition}[theorem]{Proposition}
\newtheorem{lemma}[theorem]{Lemma}

\newtheorem{corollary}[theorem]{Corollary}

\newtheorem{question}[theorem]{Question}

\theoremstyle{definition}

\newtheorem{definition}[theorem]{Definition}
\newtheorem{example}[theorem]{Example}

\newtheorem{remark}[theorem]{Remark}

\usepackage{tikz}
\usetikzlibrary{decorations.markings}

\makeatletter
\newcommand{\pushright}[1]{\ifmeasuring@#1\else\omit\hfill$\displaystyle#1$\fi\ignorespaces}
\newcommand{\pushleft}[1]{\ifmeasuring@#1\else\omit$\displaystyle#1$\hfill\fi\ignorespaces}
\makeatother

\newcommand{\ie}{\text{\it{i.e.,}}\,\,}
\newcommand{\cf}{\text{\it{cf.}}\,\,}
\newcommand{\eg}{\text{\it{e.g.,}}\,\,}

\newcommand{\CC}{\mathrm{CC}}

\newcommand{\DFA}{A^\ac}
\newcommand{\quan}{\mathfrak{A}}
\newcommand{\quanq}{\mathfrak{B}}
\newcommand{\SLH}{\mathrm{SLH}}
\newcommand{\Sym}{\Lambda}
\newcommand{\CE}{\mathrm{CE}}

\newcommand{\gl}{\mathfrak{gl}}
\newcommand{\GL}{\mathrm{GL}}
\newcommand{\SL}{\mathrm{SL}}

\newcommand{\End}{\mathrm{End}}
\newcommand{\Aut}{\mathrm{Aut}}

\newcommand{\id}{\mathrm{id}}
\newcommand{\sgn}{\mathrm{sgn}}

\newcommand*{\ihat}{\hat{\imath}}
\newcommand*{\jhat}{\hat{\jmath}}

\newcommand{\ac}{\textup{!`}}

\begin{document}

\title{Deformation and quantization of the Loday-Quillen-Tsygan isomorphism for Calabi-Yau categories 
\footnotetext{Email: xjchen@scu.edu.cn, f.eshmatov@newuu.uz, m.huang@newuu.uz}}

\author[1,2]{Xiaojun Chen}
\author[1]{Farkhod Eshmatov}
\author[1]{Maozhou Huang}

\renewcommand\Affilfont{\small}

\affil[1]{Department of Mathematics, New Uzbekistan University,
Tashkent, 100007 Uzbekistan}

\affil[2]{School of Mathematics, Sichuan University, Chengdu, Sichuan Province, 
610064 P.R. China}

\date{}

\maketitle

\begin{abstract} 
For an associative algebra $A$,
the famous theorem of Loday, Quillen and Tsygan
says that there is an isomorphism
between the graded symmetric product of the cyclic homology of $A$
and the Lie algebra homology of the infinite matrices $\mathfrak{gl}(A)$,
as commutative and cocommutative Hopf
algebras. This paper aims to study a deformation
and quantization of this isomorphism. 
We show that if $A$ is a Koszul 
Calabi-Yau algebra,
then the primitive
part
of the
Lie algebra homology
$\mathrm{H}_\bullet
(\mathfrak{gl}(A))$
has a Lie bialgebra
structure which is induced from
the Poincar\'e duality of $A$ and
deforms $\mathrm{H}_\bullet
(\mathfrak{gl}(A))$
to a co-Poisson bialgebra.
Moreover, there
is a Hopf algebra which 
quantizes
such a co-Poisson
bialgebra,
and the Loday-Quillen-Tsygan
isomorphism lifts to the quantum
level, which can be interpreted
as a quantization of the tangent
map from the tangent complex
of $\mathrm{BGL}$ to the tangent
complex of K-theory. 

\noindent{\bf Keywords:} 
Loday-Quillen-Tsygan, cyclic homology,
quantization, Calabi-Yau

\noindent{\bf MSC2020:} 16E40, 19D55
\end{abstract}

\setcounter{tocdepth}{2} 
\tableofcontents

\section{Introduction}

Let $A$ be an associative algebra over a field $k$
of characteristic zero. Let $\mathfrak{gl}(A)$ be the Lie algebra
of the infinite matrices
with entries in $A$. Denote by 
$\mathrm{H}_\bullet(\mathfrak{gl}(A))$
and
$\mathrm{HC}_\bullet(A)$
the Lie algebra homology of $\mathfrak{gl}(A)$
and
the cyclic
homology of $A$ respectively. 
The famous result of Loday and Quillen
\cite{LoQu} and Tsygan \cite{Tsygan} 
says that there is a natural
isomorphism of commutative and cocommutative graded Hopf algebras
\begin{equation}\label{iso:LQT}
\mathrm H_\bullet(\mathfrak{gl}(A))\cong
\mathbf{\Lambda}^\bullet \big(\mathrm{HC}_\bullet(A)[1]\big),
\end{equation}
where $\mathbf{\Lambda}^\bullet$ is the graded symmetric tensor product.

In the literature,
a commutative and cocommutative Hopf 
algebra is called an abelian Hopf algebra.
A natural question that arises
from the above isomorphism is:

\begin{question}\label{q:question}
Is it possible to deform the abelian
Hopf algebra
structure, in a meaningful way, 
so that the isomorphism
\eqref{iso:LQT} remains to be an isomorphism
after the deformation? If the answer is yes
(possibly for some specific algebras), 
then is it possible
to further quantize the isomorphism?
\end{question}

Before answering this question, let us
further recall the following results,
which are the main motivations for 
this paper:

\begin{enumerate}
\item[$-$] In
the paper \cite{PTVV} Pantev, To\"en,
Vaqui\'e and Vizzosi
showed that
for a (real or complex)
Calabi-Yau manifold $X$,
there is a {\it shifted symplectic
structure} on the derived
moduli stack of perfect
complexes on $X$.
Here the symplectic
structure comes from the
nowhere vanishing (real or holomorphic) volume
form, or equivalently,
the Poincar\'e duality,
on $X$, which
induces an isomorphism
from the tangent complex
to the cotangent complex
on the moduli stack, up
to a degree shifting.
Later, the above authors 
together with Calaque
continued to construct
in \cite{CPTVV}
a {\it shifted Poisson structure}
on the derived stack, and even
its quantization.
As a corollary, the functions
on the corresponding moduli stack
have a shifted Lie bracket, induced
from the shifted Poisson structure.

\item[$-$] There is an analogue
of the above results
in derived non-commutative algebraic
geometry. 
As is shown by Yeung \cite{Yeung},
Pridham \cite{Pridham2} and the
current two author's work \cite{ChEs},
for a Calabi-Yau
algebra, which is viewed
as a ``non-commutative" affine Calabi-Yau
variety, there is a ``non-commutative"
shifted
symplectic and hence Poisson 
structure on it.
These works were again motivated
by the works of Crawley-Boevey et. al. 
\cite{CBEG} and
Van den Bergh \cite{VdB}, where
a non-commutative version of
symplectic
and Poisson structures, called the
bi-symplectic and double Poisson
structures respectively,
was introduced.
Thus as a corollary,
the ``derived" functions
on a Calabi-Yau algebra $A$,
which are the cyclic
homology of $A$ by \cite{BKR},
has a shifted Lie algebra structure.
(see also \cite{ChEs}).

\item[$-$] Our third motivation
comes from algebraic K-theory.
For a long time, we have been wondering
what kind of role the Poincar\'e duality
plays in the K-theory of a Calabi-Yau
algebra, in the light of
the rich structures on its Hochschild
and cyclic homologies, as mentioned above.
In a recent work \cite{Hennion},
Hennion showed that the
tangent space to the K-theory
of a DG algebra is isomorphic
to its cyclic homology, and
the Loday-Quillen-Tsygan 
map is homotopic to the tangent
map from the tangent complex
of $\mathrm{BGL}$ to
the tangent complex of K-theory.
\end{enumerate}
Combining the above results, it is
quite natural to expect that
the answer to
Question \ref{q:question}
is affirmative, at least for 
Calabi-Yau algebras.
Our main result is:

\begin{theorem}\label{mainthm}
Let $A$ be a Koszul Calabi-Yau algebra. 
Then we have the following:
\begin{enumerate}
\item[$(1)$]\textup{(Theorem 
\ref{cor:isoHmlgyofLiecylicdeformed})}
There is a deformation of
$\mathrm{H}_\bullet(\mathfrak{gl}(A))$
to a co-Poisson bialgebra, which,
via the Loday-Quillen-Tsygan isomorphism
and Koszul duality,
is induced from the Lie bialgebra on
the cyclic homology
$\mathrm{HC}_\bullet(A^{\ac})[1]$
of the Koszul dual coalgebra $A^{\ac}$ of $A$. 

\item[$(2)$]
\textup{(Theorem \ref{cor:quantization})}
Moreover, there is a 
quantization of the above co-Poisson
bialgebra of 
$\mathrm{H}_\bullet(\mathfrak{gl}(A))$,
which therefore gives a quantization
of Hennion's tangent map to K-theory.
\end{enumerate}
\end{theorem}

The notion of co-Poisson bialgebra
was introduced by Turaev 
(see \cite{Turaev} and 
Definition \ref{def:coPoissonbialg} below),
to describe the structure
of the algebra of
loops on a Riemann surface.
In the above theorem,
$\mathrm{H}_\bullet(\mathfrak{gl}(A))$
is first identified with
$\mathbf{\Lambda}^\bullet
(\mathrm{HC}_\bullet(A)[1])$
via the Loday-Quillen-Tsygan isomorphism,
and then is identified with
$\mathbf{\Lambda}^\bullet
(\mathrm{HC}_\bullet(A^{\ac})[1])$
via Koszul duality.
Then we deform the abelian Hopf algebra $\mathbf{\Lambda}^\bullet
(\mathrm{HC}_\bullet(A^{\ac})[1])$
in the direction
of the Lie bracket
and cobracket that naturally
appear on $\mathrm{HC}_\bullet
(A^{\ac})[1]$ (see Proposition \ref{prop:Liebialg} for the construction).

The quantization of the 
co-Poisson bialgebra is inspired
by the work
\cite{Schedler} of Schedler,
where the quantization
of the necklace Lie bialgebra
of the closed paths
on a doubled quiver is constructed,
which, again, is inspired
by the work of Turaev.

\begin{remark}[Why Calabi-Yau?]
Our main construction
of the deformation
comes from the Lie bialgebra
on the {\it cyclic cohomology}
of Frobenius algebras,
or equivalently, on the 
{\it cyclic homology} of their
linear dual, co-Frobenius
coalgebras. However,
both the Loday-Quillen-Tsygan
isomorphism and Hennion's 
tangent complex are expressed
in terms of {\it cyclic homology}
of algebras.
Koszul duality solves this
discrepancy. Namely,
if a Frobenius algebra
is Koszul, then on the one hand,
its Koszul
dual algebra is a Calabi-Yau algebra,
and on the other hand, 
its cyclic cohomology (or equivalently,
the cyclic homology of its
dual co-Frobenius coalgebra) is
isomorphic to
the cyclic homology of its Koszul
dual algebra. 
This means,
the deformation of the cyclic homology 
of an algebra is induced via Koszul
duality from a deformation
of the cyclic cohomology of its Koszul
dual algebra (or equivalently
the cyclic homology of its Koszul
dual coalgebra); in the case of Calabi-Yau
algebras, such a deformation is in
general non-trivial, since it comes
from the Poncar\'e duality of the
alegbra, or equivalenlty, from its shifted bisymplectic/double Poisson
structure.

From this perspective, we also 
see that, although
the above Theorem \ref{mainthm}
is stated
for Koszul Calabi-Yau algebras,
it holds in a much broader sense,
whenever the algebra or even
the category
under study has a DG model
whose Koszul dual has 
a non-degenerate pairing
(such as cyclic $A_\infty$-algebras
or 
proper Calabi-Yau 
$A_\infty$-categories).
See \S\ref{subsect:gen}
and \S\ref{sect:application}
for more details.
We choose to work on
Koszul Calabi-Yau algebras,
just because otherwise the presentation
will be un-necessarily complicated,
which however does not improve the main
result much.
\end{remark}

\begin{remark}[Some relevant works]
This work has some overlaps with 
some previous
works. 
In \cite{CEG} 
we studied
the quantization of the Lie bialgebra
of cyclic homology of Frobenius
algebras. Such a quantization
is used in the present paper, too,
with a slight difference.
The novelty of the present paper is 
the study of the deformation and quantization
of $\mathfrak{gl}(A)$.
Although in retrospect this 
seems to be straightforward, 
we were not able to do this before.

On the other hand, in \cite{GGHZ}
Ginot, Gwilliam, Hamilton and Zeinalian
studied, in a way pretty the same to ours,
the BV quantization of the 
Loday-Quillen-Tsygan
isomorphism for 
cyclic $A_\infty$ algebras. However, their quantization
is a deformation of the Poisson algebra
in the terminology
of the current paper;
that is, they only used
the Lie algebra structure on the cyclic 
homology. Of course, they have
also obtained some 
other interesting results in the paper.
\end{remark}

The rest of this paper is devoted to the proof of the above theorem.
It is organized as follows.

In \S\ref{sect:Koszul}
we collect several necessary 
notions about Koszul algebras,
Koszul Calabi-Yau algebras,
and then recall the isomorphism
between the cyclic homology of
a Koszul algebra and
the cyclic homology of
its Koszul dual coalgebra.

In \S\ref{sect:LQT}
we first briefly go over
the Loday-Quillen-Tsygan
isomorphism for algebras,
and then go to the isomorphism
for coalgebras; the latter
is due to Kaygun \cite{Kay}.
The main result is
the isomorphism between
the Lie homology of 
$\mathfrak{gl}(A)$
and the co-Lie algebra
homology of $\mathfrak{gl}(A^{\ac})$,
which is an analogue of the 
isomorphism in the cyclic homology 
case.

In \S\ref{sect:Ktheory}
we study the algebraic K-theory
of a Koszul algebra and its Koszul dual
coalgebra; again we prove an equivalence
between them. After that, we
briefly describe Hennion's 
tangent map from $\mathrm{BGL}$
to K-theory. Here $A$ is a Koszul
algebra and $A^{\ac}$ is its Koszul
dual coalgebra.

In \S\ref{sect:deformation}
we study the deformations
of the Poisson bracket and co-Poisson
cobracket on the homology of
$\mathfrak{gl}(A)$, in the
case $A$ is a Koszul Calabi-Yau algebra.
Both deformations are carried
out via the isomorphisms 
given in \S\ref{sect:Koszul}
and \S\ref{sect:LQT}.

In \S\ref{sect:quantization}
we construct the quantization
of the Loday-Quillen-Tsygan isomorphism.
The quantization is purely
combinatorial and quite explicit,
but a bit technical.

In the last section,
\S\ref{sect:application},
we give some 
applications of our result to
geometry and topology. We show
that our main theorem is applicable
to preprojective algebra and Fukaya categories,
both have played an important role
in representation theory and mathematical
physics.

\subsection*{Acknowledgements}
This work is inspired by our study of the paper \cite{BerRam},
where some other applications
of Koszul duality and the
Loday-Quillen-Tsygan isomorphism
was explored, and
we would like to
thank Yuri Berest for
helpful communications. We
also thank Leilei Liu
for his help with the references.
This work is supported
by NSFC Nos. 12271377 and 
12261131498.

\section{Koszul Calabi-Yau
algebras}\label{sect:Koszul}

Throughout this paper,
$k$ is a field of characteristic
zero, all algebras are unital
over $k$, and all
coalgebras are co-unital over $k$
which are also 
assumed to be
co-nilpotent.

In this section, we recall some basic facts about Koszul Calabi-Yau algebras. 
In \S\ref{subsect:defKoszulalg} we recall the definition of Koszul algebras; in \S\ref{subsect:Hochs} and \S\ref{Connescycliccomplex} we study the isomorphism of the cyclic homology of an algebra with the cyclic homology of its Koszul dual coalgebra. 
In \S\ref{subsect:KoszulCY}, we focus on Koszul Calabi-Yau algebras; in particular, the Poincar\'e duality of a Calabi-Yau algebra and the non-degenerate pairing (the Frobenius algebra structure) of its Koszul dual is studied, and we give some ideas about the generalizations of these algebraic structures, which will then give more examples and applications.

\subsection{Definition of
Koszul algebra}
\label{subsect:defKoszulalg}

Suppose $V$ is a graded vector space over $k$. 
Let $R$ be a subspace of $V^{\otimes 2}$. 
Let $\langle R\rangle$ be the two sided ideal generated by $R$
in the free tensor algebra $T(V).$ 
Put $A:=T(V)/\langle R\rangle$, which is usually 
called a {\it quadratic algebra}.

Let $sV = V[-1]$ be the suspension of $V$ and put 
$A^{\ac}:=\bigoplus_n A^{\ac}_n$ with
$$
A^{\ac}_n:=\bigcap_{i+2+j=n}
(sV)^{\otimes i}\otimes(s^2R)\otimes(sV)^{\otimes j}\subseteq
(sV)^{\otimes n}.
$$
Then $A^{\ac}$ is a graded coalgebra 
whose coproduct is
induced from the de-concatenation of $T(sV)$.
Choose a basis $\{e_i\}$ of $V$
and let $\{e_i^*\}$ be the dual basis of $V^*$. Then
$d:=\sum_i e_i\otimes s^{-1}e_i^*$ acts on
$A\otimes_k A^{\ac}$ by
$$
d(a\otimes f):=\sum_i 
ae_i\otimes s^{-1}e_i^*\cdot f,$$
which is independent of the choice of the basis $\{e_i\}.$
The equation $d^2=0$ holds as elements in $\langle R \rangle$ vanish in $A$.
We thus get a chain complex, called the
{\it Koszul complex} of $A$:
\begin{equation}\label{eq:Koszulcomplex}
\cdots\stackrel{d}\longrightarrow
A\otimes_k A^{\ac}_n\stackrel{d}\longrightarrow
A\otimes_k A^{\ac}_{n-1}\stackrel{d}\longrightarrow\cdots
\stackrel{d}\longrightarrow
A\otimes_k A^{\ac}_0\stackrel{\varepsilon}\longrightarrow k,
\end{equation}
where $A^{\ac}_0=k$ and $\varepsilon$ is the augmentation of $A$.

\begin{definition}
\label{def:Koszulalg}
The algebra $A$ is called a {\it Koszul algebra} if the Koszul complex
\eqref{eq:Koszulcomplex}
is acyclic.
In this case, $A^{\ac}$ is called the {\it Koszul dual coalgebra}
of $A$, and its linear dual, $\mathrm{Hom}_k(A^{\ac}, k)$, denoted
by $A^!$, is called the {\it Koszul dual algebra} of $A$.
\end{definition}

\subsection{Hochschild and 
cyclic homology}
\label{subsect:Hochs}

In this subsection we study
the Hochschild and cyclic homology of Koszul algebras.
In the rest of the paper, the bar construction of an algebra $A$ (resp. the cobar construction of a coalgebra $C$) is denoted $\mathrm{B}(A)$ (resp. $\Omega(C)$). 
If $A$ is a Koszul algebra, then 
by \eqref{eq:Koszulcomplex} we have
$\mathrm{Ext}_A^\bullet(k,k)\cong A^!$
as a graded algebra.
As a corollary, we have the following 
(see \cite{LoVa}):

\begin{proposition}\label{prop:qisofcobarresl}
Let $A$ be a Koszul algebra and $A^{\ac}$ its Koszul dual coalgebra. 
Then there
is a natural DG algebra 
quasi-isomorphism
$p: \Omega(A^{\ac})\to A$,
which gives a cofibrant (that is, quasi-free) resolution of $A$.
\end{proposition}

\begin{definition}
\label{def:Hochschild}
(1) For a (graded) algebra $A$, its 
\emph{cyclic complex} $\mathrm{CC}_\bullet(A)$ 
is the total complex of the following (reduced)
mixed Hochschild complex 
$(\mathrm{CH}_\bullet(A),b, B)$, where 
$\mathrm{CH}_\bullet(A)=A\otimes\mathrm B(A)$ 
and
\begin{eqnarray*}
b(a_0\otimes[a_1,\cdots,a_n])
&=&(-1)^{|a_0|}a_0a_1
\otimes[a_2,\cdots,a_n]\\
&&
+\sum_{i=1}^{n-1}
(-1)^{|a_0|+
\cdots+|a_{i}|+i}a_0\otimes[a_1,\cdots,
a_{i}a_{i+1},\cdots, a_n]\\
&&-(-1)^{(|a_n|+1)(|a_0|+\cdots+|a_{n-1}|+n-1)}a_na_0\otimes[a_1,\cdots,a_{n-1}],\\
B(a_0\otimes[a_1,\cdots,a_n])&=&
\sum_{i=0}^n(-1)^{\sigma}\cdot
1\otimes[a_i,\cdots,a_n, a_0,\cdots, a_{i-1}],
\end{eqnarray*}
where
$\sigma =
(|a_0|+\cdots+|a_{i-1}|+i)
(|a_i|+\cdots+|a_n|+n-i+1)$.

(2) For a (graded) coalgebra $C$, its \emph{cyclic complex} $\mathrm{CC}_\bullet(C)$ is the total complex of the following mixed Hochschild complex $(\mathrm{CH}_\bullet(C),b, B)$, where $\mathrm{CH}_\bullet(C)=C\otimes\Omega(C)$ and 
\begin{eqnarray}
b(c_0\otimes[c_1,\cdots,c_n])
&=&\sum_{(c_0)}(-1)^{|a_0'|}
c_0'\otimes[c_0'', c_1,\cdots,
c_n]\nonumber\\
&&+\sum_{i=1}^n\sum_{(c_i)}
(-1)^{|c_0|+\cdots+|c_{i-1}|
+|c_i'|-i}c_0\otimes[c_1,\cdots,
c_i', c_i'',\cdots, c_n]\nonumber\\
&&-\sum_{(c_0)}
(-1)^{(|c_0'|-1)(|c_0''|+|c_1|+
\cdots+|c_n|-n)}
c_0''\otimes[c_1,\cdots, c_n, c_0'],
\nonumber\\
B(c_0\otimes[c_1,\cdots,c_n])
&=&
\sum_{i=1}^n(-1)^\nu\varepsilon(c_0)
\cdot
c_i\otimes[c_{i+1},\cdots, 
c_n,c_1,\cdots, c_{i-1}],\nonumber
\end{eqnarray}
where 
$\nu=(|c_1|+\cdots+|c_{i-1}|-(i-1))
(|c_i|+\cdots+|c_n|-(n-i+1))$
and $\varepsilon:C\to k$ is the co-unit.
Here the sum $\sum_{(c)}$ extends over summands in the coproduct of $c$.

(3) The associated homologies of $\CC_\bullet(A)$ and $\CC_\bullet(C)$ are called the {\it cylic homologies} of $A$ and $C$, and are denoted by $\mathrm{HC}_\bullet(A)$ and $\mathrm{HC}_\bullet(C)$ respectively.
\end{definition}

The signs in the above definition 
are given by the Koszul
sign rule; that is, whenever
two graded objects (such
as elements in a graded space
or graded operators or a mixture
of them) exchange, there arises
a sign which is the product
of their parities. In other words,
for two graded space $V$ and $W$,
the canonical 
isomorphism $V\otimes W\to
W \otimes V$ is given by
$v\otimes w\mapsto(-1)^{|v|\cdot|w|}w\otimes v$. In what follows, we sometimes omit
the precise value of the sign, and just write $\pm$ instead.

The following proposition relates
the cyclic complexes of $A$
and $A^{\ac}$:

\begin{proposition}
\label{prop:qisforKoszul}
Let $A$ be a Koszul algebra and $A^{\ac}$ its Koszul dual coalgebra.
Then we have an quasi-isomorphism of mixed complexes
$$(\mathrm{CH}_\bullet(A),b, B)
\simeq(\mathrm{CH}_\bullet(A^{\ac}),b,B).$$
\end{proposition}

\begin{proof}[Sketch of proof] 
See \cite{CYZ} for a detailed proof. 
We show that the quasi-isomorphism
is given by the following quasi-isomorphisms
of mixed complexes:
\begin{equation}\label{qis:twoqis}
\begin{split}\xymatrix{
&(\mathrm{CH}_\bullet
(\Omega(A^{\ac})), b, B)\ar[rd]^{\mathrm{CH}(p)}
&\\
(\mathrm{CH}_\bullet(A^{\ac}), b, B)
\ar[ur]^{\id\otimes\iota}
&&
(\mathrm{CH}_\bullet(A), b, B)
}
\end{split}
\end{equation}
For the left arrow, let $\overline{\Delta}: A^{\ac}\to A^{\ac}\otimes
A^{\ac}$ be the {\it reduced} 
coproduct of $A^{\ac}$,
and let
$$\overline{\Delta}^n=(\overline{\Delta}\otimes \id^{\otimes n-1})\circ(\overline{\Delta}\otimes \id^{\otimes n-2})\circ\cdots\circ\overline{\Delta}:
A^{\ac}\to (A^{\ac})^{\otimes n+1},\quad n\ge 1.$$
Let $\iota: A^{\ac}
\to \mathrm B\Omega(A^{\ac})$
be the coalgebra extension of
the natural map
$A^{\ac}\twoheadrightarrow \overline{A}^{\ac}\subset
\Omega(A^{\ac})[1]$ (here
$\overline{A}^{\ac}$ is the kernel of
the co-unit: $\varepsilon: A^{\ac}\to k$); 
that is,
for any $c\in A^{\ac}$,
$$\iota(c):=c+\overline{\Delta}(c)
+\overline{\Delta}^2(c)+\overline{\Delta}^3(c)+\cdots
\in \mathrm B\Omega(A^{\ac}).$$
Then it is shown by
Jones-McCleary \cite[\S6]{JoMc}
(see also \cite[Theorem 15]{CYZ})
that
\begin{equation}\label{qis:JoMc}
\id\otimes\iota:
(\mathrm{CH}_\bullet(A^{\ac}), b, B)\to
(\mathrm{CH}_\bullet
(\Omega(A^{\ac})),b,B)
\end{equation}
is a quasi-isomorphism of mixed complexes.

For the right arrow, since 
$p:\Omega(A^{\ac})\to A$
is a quasi-isomorphism (by Proposition
\ref{prop:qisofcobarresl}), the induced map
on the Hochschild chain complexes
\begin{equation}\label{qis:GJP}\mathrm{CH}(p):
(\mathrm{CH}_\bullet
(\Omega(A^{\ac})),b,B)
\to
(\mathrm{CH}_\bullet(A),b,B)
\end{equation}
is again a quasi-isomorphism of mixed complexes,
by an easy argument of spectral sequences.

Combining the above two 
quasi-isomorphisms
\eqref{qis:JoMc} and \eqref{qis:GJP},
we get \eqref{qis:twoqis}.
\end{proof}

\subsection{The Connes
cyclic complex}\label{Connescycliccomplex}

In what follows 
we shall also use the Connes
cyclic complex of an algebra
and a coalgebra. Let
us briefly recall their construction.

Let $A$ be a DG algebra
and $C$ be a DG coalgebra.
For any $n\ge 0$, let
$t: A^{\otimes n+1}\to
A^{\otimes n+1}$
be the cyclic permutation
$$t(a_0,a_1,a_2,\cdots, a_n)
:=(-1)^{(|a_n|+1)(|a_0|+\cdots+|a_{n-1}|+n)}
(a_n, a_0, \cdots, a_{n-1}),$$
and similarly, let
$t: C^{\otimes n+1}\to
C^{\otimes n+1}$
be the cyclic permutation
$$t(c_0,c_1,c_2,\cdots, c_n)
:=(-1)^{(|c_0|-1)(|c_1|+\cdots+|c_{n}|-n)}
(c_1, \cdots, c_{n}, c_0).$$
Let $N:=1+t+\cdots+t^n$ in both cases.
It is easy to see that
$(1-t)N=N(1-t)=0$.
Now 
consider the following two
(usual) cyclic complexes
\begin{equation}
\label{def:totalcomplex}
\begin{split}
\xymatrix{
\ar[d]&\ar[d]&\ar[d]&\\
A^{\otimes 3}
\ar[d]_b&\ar[l]_-{1-t}
A^{\otimes 3}\ar[d]_{b'}
&\ar[l]_-N
A^{\otimes 3}\ar[d]_b&\ar[l]_-{1-t}\\
A^{\otimes 2}
\ar[d]_b&A^{\otimes 2}
\ar[d]_{b'}\ar[l]_-{1-t}
&A^{\otimes 2}\ar[d]_b
\ar[l]_-N
&\ar[l]_-{1-t}\\
A&
\ar[l]_-{1-t}
A&
\ar[l]_-N
A&\ar[l]_-{1-t}
}
\quad\quad
\xymatrix{
&&&\\
C^{\otimes 3}
\ar[u]\ar[r]^-{1-t}&
C^{\otimes 3}\ar[u]
\ar[r]^-N&
C^{\otimes 3}\ar[u]\ar[r]^-{1-t}&\\
\ar[u]^{b}\ar[r]^-{1-t}
C^{\otimes 2}
&
\ar[r]^-N\ar[u]^{b'}
C^{\otimes 2}
&C^{\otimes 2}\ar[u]^{b}
\ar[r]^-{1-t}&\\
C\ar[r]^-{1-t}\ar[u]^{b}
&
C\ar[r]^-{N}\ar[u]^{b'}
&
C\ar[u]^{b}
\ar[r]^-{1-t}&
}
\end{split}
\end{equation}
Here in both complexes, 
the boundaries $b$ are the
same $b$ as in Definition
\ref{def:Hochschild},
but $b'$ 
only contain the first
two summands of $b$.
Then the {\it Connes cyclic
complex} of $A$, denoted by
$\mathrm{CC}_\bullet^{\lambda}(A)$,
is the complex
$$
\mathrm{CC}_\bullet^{\lambda}(A)
:=\bigoplus_{n\ge 0}
\mathrm{CC}_n^{\lambda}(A)
=\bigoplus_{n\ge 0}
A^{\otimes n+1}/1-t,
$$
and the {\it Connes cyclic complex}
of $C$, denoted by
$\mathrm{CC}_\bullet^\lambda(C)$,
is the complex below with boundary map $b$.
$$
\mathrm{CC}_\bullet^{\lambda}(C)
:=\bigoplus_{n\ge 0}
\mathrm{CC}_n^{\lambda}(C)
=\bigoplus_{n\ge 0}N(C^{\otimes n+1}).
$$

It is proved in
Loday \cite[\S2.1.11]{Loday}
that for any unital algebra,
the total complex in
Definition \ref{def:Hochschild}
is quasi-isomorphic to
the total complex \eqref{def:totalcomplex}
and also to the Connes cyclic complex.
It is also true for co-unital
coalgebras by the same argument.
Thus combining these
quasi-isomorphisms with
Proposition 
\ref{prop:qisforKoszul}
we obtain the following.

\begin{proposition}
\label{prop:qisinKoszulCY}
Let $A$ be a Koszul algebra and $A^{\ac}$ its Koszul dual coalgebra. 
Then we have quasi-isomorphisms
\begin{equation}\label{qis:Koszuldualityforalg}
\begin{split}
\xymatrixcolsep{4pc}
\xymatrix{
&\mathrm{CC}_\bullet^\lambda
(\Omega(A^{\ac}))\ar[rd]^{\simeq}
&\\
\mathrm{CC}^\lambda_\bullet(A^{\ac})
\ar[ur]^{\simeq}
&&
\mathrm{CC}^\lambda_\bullet(A).
}
\end{split}
\end{equation}
\end{proposition}

For an algebra $A$ or a coalgebra $C$,
one can also consider 
their cyclic cochain complexes,
which are the linear dual complexes
of the cyclic complexes.
We denote them by $\mathrm{CC}^\bullet_\lambda(A)$
and $\mathrm{CC}^\bullet_\lambda(C)$
respectively.
In particular,
if $A^{!}$ is 
a finite dimensional algebra
and $A^{\ac}$ is its linear dual
coalgebra,
then 
$\mathrm{CC}_\bullet^\lambda(A^{\ac})
\cong
\mathrm{CC}_\lambda^\bullet(A^!)$
canonically. In particular,
if furthermore $A^!$ is the Koszul dual
algebra of an algebra $A$, 
then combining
it with \eqref{qis:Koszuldualityforalg}
we get 
quasi-isomorphisms
\begin{equation}
\mathrm{CC}^\lambda_\bullet(A)
\stackrel{\eqref{qis:Koszuldualityforalg}}
\simeq
\mathrm{CC}_\bullet
^\lambda(A^{\ac})
\cong
\mathrm{CC}_\lambda^\bullet(A^!).
\end{equation}

\begin{remark}[Periodic and negative cyclic
homology]\label{rem:negativecyclic}
We may extend the cyclic complex of $A$ given in
\eqref{def:totalcomplex}
to be a complex
in the full first and second quadrant,
where the left-most column is placed
on the $y$-axis.
The associated homology is called
the {\it periodic cyclic homology} of $A$.
The sub-chain complex
in the second quadrant (including
those terms on the $y$-axis), denoted
by $\mathrm{CC}_\bullet^-(A)$,
is called
the {\it negative cyclic complex} of $A$,
and its homology is called the
{\it negative cyclic homology} of $A$. 
\end{remark}

\subsection{Koszul Calabi-Yau algebras}
\label{subsect:KoszulCY}

In this subsection, we study
with some details a special class
of Koszul algebras, the Koszul Calabi-Yau algebras.
The notion of Calabi-Yau algebra is introduced
by V. Ginzburg in \cite{Ginzburg2}. It has been
widely studied in recent years.

\begin{definition}[\text{\cite[\S3.2.3]{Ginzburg2}}]\label{def:CYalg}
For an associative $k$-algebra $A$, it is called an \emph{$n$-Calabi-Yau algebra}
(or \emph{Calabi-Yau algebra of dimension $n$})
if it is homologically smooth
and
$$
\mathrm{RHom}_{A\otimes A^{op}}(A, A\otimes A)\cong
A[n]
$$
in the derived category 
$D(A\otimes A^{op})$ of $A$-bimodules.
\end{definition}

\subsubsection{Non-commutative 
Poincar\'e duality}
\label{subsubsect:PD}

For a Calabi-Yau algebra $A$ of dimension $n$, there is a class $\omega\in\mathrm{HH}_n(A)$ with
\begin{equation}\label{eq:VdBduality}\mathrm{HH}^\bullet(A)\stackrel{\cap\omega} \longrightarrow\mathrm{HH}_{n-\bullet}(A)
\end{equation} 
is an isomorphism, where $\cap\omega$ means
the cap product.
It is called the non-commutative Poincar\'e
duality of $A$, due to Van den Bergh.
In the case of Koszul Calabi-Yau algebras,
one can write down the volume class $\omega$ explicitly. 

Let us start with the following proposition,
which is proved in \cite{Ginzburg2}:

\begin{proposition}\label{prop:PDofCY}
If $A$ is a Koszul Calabi-Yau algebra of dimension $n$, then its Koszul dual algebra
$A^!$ is a graded Frobenius algebra of degree $n$.
\end{proposition}

Recall that an algebra $A^!$ is called a {\it Frobenius algebra}
if there is a bilinear, symmetric
non-degenerate pairing $\langle-,-\rangle: A^!\times A^!\to k$
such that
\begin{equation}
\langle a, b\cdot c\rangle=\langle a\cdot b, c\rangle,\quad\mbox{for all}\,
a,b\in A^{\ac}.\label{FrobA}
\end{equation}
A graded Frobenius algebra of degree $n$ means
a graded Frobenius algebra with the pairing of degree $n$, \ie the degree of the output is lowered by $n$ compared with the inputs.

\begin{proof}[Sketch proof
of Proposition~\text{\rm \ref{prop:PDofCY}}]
Since $A$
is Koszul, it is straightforward that
$$
A\otimes A^{\ac}_n
\otimes A\longrightarrow
\cdots\longrightarrow
A\otimes A^{\ac}_0\otimes A
\stackrel{\mu}\longrightarrow A
$$
is a resolution of $A$
as $A$-bimodules. Here
$A^{\ac}_0\cong k$
and $\mu: A\otimes
A^{\ac}_0\otimes A\cong
A\otimes A\to A$
is the product map.
Now since $\mathrm{RHom}_{A\otimes A^{op}}(A,
A\otimes A)$
is isomorphic to
$$
A\otimes A^!_0
\otimes A\longrightarrow
\cdots\longrightarrow
A\otimes A^!_n,
$$
the isomorphism
$\mathrm{RHom}_{A\otimes A^{op}}(A,
A\otimes A)\cong A[n]$
implies that
$A^{\ac}_i\cong A^!_{n-i}$. 
Moreover, the preimage of $1\in A_0^!$, denoted by $\omega^{\ac}\in A^{\ac}_n$, gives inner product on $A^!$: $$(u,v)\mapsto \omega^{\ac}(u\cup v),\quad \mbox{for all}\, u, v\in A^!.$$ This gives the Frobenius algebra structure on $A^!$.
\end{proof}

Moreover,
since $\mathrm{HH}_\bullet(A)\cong
\mathrm{HH}_\bullet(A^{\ac})$,
the preimage of $\omega^{\ac}$
under this isomorphism
is exactly the volume class
for the Poincar\'e duality 
\eqref{eq:VdBduality}
of $A$.

Now if $A^!$ is 
a graded Frobenius algebra of degree $n$, 
the its linear dual,
denoted by $A^{\ac}$, 
is a graded coalgebra, and moreover 
$A^{\ac}$
has an induced product from the product on
$A^!$ via the non-degenerate pairing. 
Let $\bullet$ and $\Delta$ denote
the product and coproduct of $A^{\ac}$ 
respectively.
Then by \eqref{FrobA}, one obtains
\begin{equation}\label{eq:coproduct}
\Delta(a\bullet b)=a\bullet\Delta(b)=\Delta(a)\bullet b,\quad\mbox{for all}\,
a,b\in A^{\ac}.
\end{equation}
In the following, 
we call $A^{\ac}$
a {\it co-Frobenius coalgebra}
for convenience. 

\subsubsection{Some 
generalizations}\label{subsect:gen}

Calabi-Yau algebras and Frobenius
algebras are closed related
to the so-called Calabi-Yau categories.
The notion of Calabi-Yau categories
were probably first introduced by 
Kontsevich (\cite{Kon},
see also \cite{Cos}).
Originally by a Calabi-Yau category,
or more generally, a Calabi-Yau
$A_\infty$-category,
people mean a category
with a cyclic non-degenerate pairing
on its morphism space.
Later various new
Calabi-Yau structures have been
introduced; see, for example,
\cite{BraDyc1,Gan,KoTaVl}.
In the following, we 
adopt the notion which is most
relevant to us.

\begin{definition}\label{def_Ainfty}
An \emph{$A_\infty$ category} $\mathcal A$ over $k$ consists of a set of objects $\mathrm{Ob}(\mathcal A)$,
a graded $k$-vector space
$\mathrm{Hom}(A_1,A_2)$ for each pair of objects $A_1,A_2\in\mathrm{Ob}(\mathcal A)$, and a sequence of
multilinear maps for $n=1,2,\ldots$
$$m_n: \mathrm{Hom}(A_{n},A_{n+1})
\otimes\cdots\otimes 
\mathrm{Hom}(A_2, A_3)\otimes \mathrm{Hom}
(A_1,A_2)\to \mathrm{Hom}(A_1,A_{n+1}), $$
with degree $|m_n|=2-n$ and satisfying the following $A_\infty$ relations:
\begin{equation}\label{higher_htpy}
\sum_{p=1}^{n}\sum_{k=1}^{n-p+1}(-1)^{\mu_{p}}
m_{n-k+1}(a_n,\cdots,a_{p+k},
m_{k}(a_{p+k-1},\cdots,a_{p}),
a_{p-1},\cdots,a_1)=0,
\end{equation}
where $a_i\in\mathrm{Hom}(A_{i},A_{i+1})$, 
for $i=1,2,\ldots,n$, and
$\mu_{p}=|a_1|+\cdots+|a_{p-1}| - (p-1).$
\end{definition}

In the above definition,
if all $m_n$ ($n\ge 3$) vanish,
then $\mathcal A$ is the usual
DG category; if furthermore,
the category has only one object,
then it is nothing but a DG algebra.

For an $A_\infty$ category, we can
define its Hochschild homology
and cyclic homology similarly
to the case of algebras,
by encapsulating those
higher homotopy operators $m_n$.
The interested reader may refer
to the above cited papers for more
details.

An $A_\infty$ category $\mathcal A$
is called {\it proper} if for any two
objects $X$ and $Y$, $\mathrm{H}^\bullet(
\mathrm{Hom}_{\mathcal A}(X,Y))$
is finite dimensional.
It is called {\it homologically smooth}
if the diagonal $\mathcal{A}$-bimodule
$\mathcal A_\Delta$,
which assigns to any two objects
$X$ and $Y$ the chain complex
$\mathcal A_\Delta(X,Y):=\mathrm{Hom}_{\mathcal A}(Y, X)$,
is a perfect $\mathcal A$-bimodule.

\begin{definition}[{Proper Calabi-Yau category, see \cite[Definitions 6.1 \& 6.3]{Gan}}]
(1) Let $\mathcal A$ be a proper
$A_\infty$-category.
It is called {\it weak proper $n$-Calabi-Yau} if
there is a chain map
$\mathrm{tr}:\mathrm{CH}_n(\mathcal A)
\to k$
such that for any two objects $X$ and $Y$
in $\mathcal A$,
the composition
$$
\mathrm{H}^\bullet(
\mathrm{Hom}_{\mathcal A}(X,Y))
\otimes
\mathrm{H}^{n-\bullet}(
\mathrm{Hom}_{\mathcal A}(Y,X))
\stackrel{[m_2]}\longrightarrow
\mathrm{H}^n(
\mathrm{Hom}_{\mathcal A}(Y,Y))
\stackrel{[i]}\longrightarrow
\mathrm{H}^\bullet(\mathcal A)
\stackrel{[tr]}\longrightarrow k
$$
is a perfect pairing,
where $[i]$
is induced from the natural inclusion
$i:\mathrm{Hom}_{\mathcal A}(Y,Y)\to
\mathrm{CH}_\bullet(\mathcal A)$.

(2) If $\mathcal A$ is a weak proper Calabi-Yau 
category, then it is called {\it (strong) proper Calabi-Yau}
if there is a chain map
$\widetilde{\mathrm{tr}}: \mathrm{CC}_n(\mathcal A)\to k$
such that $\mathrm{tr}=\widetilde{\mathrm {tr}}\circ\mathrm{pr}$, where
$\mathrm{pr}:\mathrm{CH}_\bullet(\mathcal A)
\to\mathrm{CC}_\bullet(\mathcal A)$ is the natural projection (see 
\eqref{def:totalcomplex}).
\end{definition}

\begin{definition}[Smooth Calabi-Yau category]
(\cite[Definition 6.5]{Gan})
(1) Let $\mathcal A$ be a homologically smooth
$A_\infty$-category.
A cycle $\omega\in\mathrm{CH}_{n}(\mathcal A)$
is said to be a {\it weak smooth Calabi–Yau structure} if the cap product
with $\omega$
induces an isomorphism on homology level:
$$[\cap\omega]:
\mathrm{HH}^\bullet(\mathcal A)
\cong\mathrm{HH}_{n-\bullet}(\mathcal A)
$$

(2) If $\mathcal A$ is a weak smooth Calabi-Yau
category, then it is called {\it (strong) smooth Calabi-Yau} if there is a cycle $\widetilde{\omega}
\in\mathrm{CC}_n^{-}(\mathcal A)$
such that $\omega=\iota(\widetilde{\omega})$,
where $\iota:\mathrm{CC}_\bullet^{-}(\mathcal A)
\to \mathrm{CH}_\bullet(\mathcal A)$
maps a negative cyclic chain
to its right-most column as described in
Remark \ref{rem:negativecyclic}.
\end{definition}

It is quite straightforward
to check that a Frobenius
algebra is a proper Calabi-Yau
category
with one object, while
a Koszul Calabi-Yau algebra
is a smooth Calabi-Yau
category with one object.

Holstein-Lazarev introduced
in \cite{HolLaz}
the notion of
Koszul duality
for categories.
Given a (small) category $\mathcal A$,
its Koszul dual is a coalgebra.
Roughly speaking,
the Koszul dual coalgebra
$\mathcal A^{\ac}$
is the bar construction
of the Hom-space
of $\mathcal A$ over
$k[\mathrm{Ob}(\mathcal A)]$,
which is denoted by 
$\mathrm{B}\mathcal A$;
moreover, they showed
that the functor
$\mathcal A\to\mathcal A^{\ac}$
passes to homotopy categories,
which is an equivalence.
Thus we may view any
coalgebra which is quasi-isomorphic
to $\mathcal A^{\ac}$
as the Koszul dual of $\mathcal A$.
The following result then generalizes Proposition 
\ref{prop:qisinKoszulCY}.

\begin{proposition}
Suppose $\mathcal A$ is 
small DG category,
and denote is Koszul dual 
coalgebra by
$\mathcal A^{\ac}$. 
Then there is 
a quasi-isomorphism
$$\mathrm{CC}_\bullet^\lambda(\mathcal A^{\ac})\stackrel{\simeq}\longrightarrow
\mathrm{CC}_\bullet^\lambda(\mathcal A).$$
\end{proposition}

\begin{proof}
In \cite[Lemma 3.39]{HolLaz}, 
the authors showed that
there is a quasi-equivalence
$\Omega\mathrm B\mathcal A\simeq\mathrm A$.
In the proof of Proposition
\ref{prop:qisforKoszul},
replacing
$A^{\ac}$ by $\mathrm B\mathcal A$, 
and replacing 
$\Omega(A^{\ac})$ by
$\Omega\mathrm B\mathcal A$, 
we get the desired quasi-isomorphism.
\end{proof}

There are many examples of
Calabi-Yau algebras which
are Koszul in the sense
of Definition 
\ref{def:Koszulalg},
such as the 3- and 4-dimensional
Sklyanian algebras.
In what follows, we give two
examples of Koszul Calabi-Yau algebras
in the above broader sense.
In \S\ref{sect:application}
we shall give more examples
of Calabi-Yau algebras/categories
and their applications.

\begin{example}[Universal enveloping
algebras]
Suppose $\mathfrak g$ is a Lie algebra, then the universal enveloping algebra $U(\mathfrak g)$ is Koszul in the above sense (strictly speaking it is {\it linear-quadratic Koszul}; see \cite{LoVa}), and its Koszul dual differential graded coalgebra is exactly the Chevalley-Eilenberg complex $(\mathrm{CE}_\bullet(\mathfrak g), d)$. 
The top chain of the Chevalley-Eilenberg complex is a cycle if and only if $\mathrm{Tr}(\mathrm{ad}(g))=0$ for all $g\in\mathfrak g$. 
A Lie algebra satisfying this property is called {\it unimodular}, such as abelian Lie algebras, semi-simple Lie algebras, Heisenberg Lie algebras, and the Lie algebra of compact groups. 
He, Van Oystaeyen, and Zhang proved in (\cite[Theorem 5.3]{HVZ}) that for an $n$-dimensional Lie algebra $\mathfrak{g}$, its universal enveloping algebra $U(\mathfrak g)$ is $n$-Calabi-Yau if and only if $\mathfrak g$ is unimodular.
\end{example}

\begin{example}[Yang-Mills algebras]
Yang-Mills algebras were introduced by Connes and Dubois-Violette (\cite{CDV}).
An algebra is called a \emph{Yang-Mills algebra} of $n$ generators and denoted by
$\mathrm{YM}(n)$, if it is generated by
the elements $x_i$ ($i\in\{ 1,\cdots,n\}$) with the following relations:
$$
g^{ij}[x_i,[x_j,x_l] ] =0, \, l\in \{ 1,2,\cdots,n\},
$$
where $(g^{ij})$ is a symmetric invertible $n\times n$-matrix.
In loc. cit. the authors also showed
that for all $n$, 
$\mathrm{YM}(n)$ is Koszul
in the above sense, and its Koszul
dual algebra only has 
$m_2$ and $m_3$, and all other $m_n$'s
vanish. Later,
Berger-Taillefer 
(\cite[Proposition 4.5]{BT}) proved 
$\mathrm{YM}(n)$ is $3$-Calabi-Yau.

We next give a description of the 
proper
$A_\infty$ structure on $\mathrm{YM}(n)^!$.
For simplicity, we take $(g^{ij})$ to be 
the identity matrix, and
denote $V=\mathrm{Span}_k\{x_1,x_2,\cdots,x_n\}$.
Let
$$
\begin{array}{lll}
\mathrm{YM}(n)^{\vee}_0=k1,& \mathrm{YM}(n)^{\vee}_1=V^*,&\mathrm{YM}(n)^{\vee}_2=\bigoplus_{i,j=1}^n kx_i^*x_j^*,\\
\mathrm{YM}(n)^{\vee}_3=\bigoplus_{i=1}^n kx_i^*z,
& \mathrm{YM}(n)^{\vee}_4=k z^2,&\mathrm{YM}(n)^{\vee}_i=0 \quad(i>4),
\end{array}
$$
where
$z=\sum_{i=1}^n (x_i^*)^2$.
Then the Koszul dual algebra of $\mathrm{YM}(n)$ is given as follows (see \cite[Proposition 1]{CDV}):
\begin{eqnarray*}
\mathrm{YM}(n)^{!}_0=\mathrm{YM}(n)^{\vee}_0,&&
\mathrm{YM}(n)^{!}_1=\mathrm{YM}(n)^{\vee}_1,\\
\mathrm{YM}(n)^{!}_2=\mathrm{YM}(n)^{\vee}_3,&&
\mathrm{YM}(n)^{!}_3=\mathrm{YM}(n)^{\vee}_4,
\end{eqnarray*}
The $A_\infty$ operators $m_2$ and $m_3$ are given as follows:
denoting
the above identification by $\phi:
\mathrm{YM}(n)^{!}_i\stackrel{\cong}{\to} \mathrm{YM}(n)^{\vee}_j,$
then
for $f_1, f_2, f_3\in\mathrm{YM}(n)^!$,
\begin{equation*}
m_2(f_1,f_2)=
\left\{
\begin{array}{cl}
\phi^{-1}(\phi(f_1)\cdot\phi(f_2)),
&\mbox{if}\; \phi(f_1)\cdot\phi(f_2)\notin \mathrm{YM}(n)_2^{\vee},\\
0,&\mbox{otherwise,}
\end{array}
\right.
\end{equation*}
and
\begin{equation*}
m_3(f_1,f_2,f_3)=
\left\{
\begin{array}{cl}
\phi^{-1}(\phi(f_1)\cdot\phi(f_2)\cdot\phi(f_3)),&\mbox{if}\;  \phi(f_1)\cdot\phi(f_2)\cdot\phi(f_3)\notin \mathrm{YM}(n)_2^{\vee},\\
0,&\mbox{otherwise.}
\end{array}
\right.
\end{equation*}
The non-degenerate pairing on $\mathrm{YM}(n)^{!}$ is
the coefficient of $z^2$
in the corresponding product.
\end{example}

\section{The Loday-Quillen-Tsygan
isomorphism}\label{sect:LQT}

In this section, we recall Loday-Quillen-Tsygan's isomorphism for DG algebras, and then propose its coalgebra analogue (the latter isomorphism has also been obtained by Kaygun \cite{Kay}). Our main result is Theorem~\ref{cor:isoHmlgyofLiecylic} and the main reference is Loday's book \cite{Loday}.

\subsection{The isomorphism for algebras}

Let $A$ be a (DG) associative algebra. Let $M_r(A)$ be the set of $r\times r$ matrices with entries in $A$. $M_r(A)$ has a natural Lie algebra structure given by the commutators, and moreover, $M_r(A)$ embeds in $M_{r+1}(A)$ in the upper-left diagonal as Lie algebras. Let $$\mathfrak{gl}(A):={\lim_{\longrightarrow}}_rM_r(A).$$ 
It has a sub-Lie algebra $\gl(k)$.

In the following, for a Lie algebra $\mathfrak{g}$, we write $\mathrm{CE}_\bullet(\mathfrak{g})$ to be its Chevalley-Eilenberg chain complex. By \cite[10.1.8 \& 10.2.9]{Loday}, there exists a decomposition of chain complex
\begin{equation}\label{iso:decompCE1}
\mathrm{CE}_\bullet
(\mathfrak{gl}(A))
\cong
(\mathrm{CE}_\bullet
(\mathfrak{gl}(A)))_{\mathfrak{gl}(k)}\bigoplus
L_*,
\end{equation}
where $L_*$ is an acyclic complex.
Furthermore,
\begin{equation}\label{iso:decompCE2}
(\mathrm{CE}_\bullet
(\mathfrak{gl}(A)))_{\mathfrak{gl}(k)}
\cong\bigoplus_n (k[S_n]\otimes
A^{\otimes n})_{S_n}.
\end{equation}
Let $U_n\subset S_n$ be the conjugacy
class of the cycle $\tau_n = (12\cdots n).$ 
Then we have:

\begin{proposition}[{\cite[Proposition 10.2.16]{Loday}}]
\label{prop:primitive}
The above $\bigoplus_n (k[S_n]\otimes A^{\otimes n})_{S_n}$ has a DG commutative and cocommutative Hopf algebra structure. The set of primitive elements is $$\bigoplus_n (k[U_n]\otimes A^{\otimes n})_{S_n},$$ which is further isomorphic to the cyclic chain complex $\mathrm{CC}_\bullet^\lambda(A)$.
\end{proposition}

In fact, the second isomorphism is given by
\begin{equation}\label{eq:maptocycliccpx}
(k[U_n]\otimes A^{\otimes n})_{S_n}
\stackrel{\cong}{\longrightarrow}
k\otimes_{k[S_n]}k[S_n/\mathbb Z_n]\otimes
A^{\otimes n}
\stackrel{\cong}{\longrightarrow}
k\otimes_{k[\mathbb Z_n]}(k\otimes A^{\otimes n})
\stackrel{\cong}{\longrightarrow}
(A^{\otimes n})_{\mathbb Z_n},
\end{equation}
where $U_n$ is isomorphic to $S_n/\mathbb Z_n$ as $S_n$-sets. This map is compatible with the usual trace map, which we now recall.

\begin{definition}
Let $\theta:
\mathrm{CE}_\bullet(\mathfrak{gl}(A))\to
\mathrm{CC}_{\bullet-1}^\lambda(M(A))$
be
$$\theta(a_0\wedge \cdots\wedge a_n)
=\sum_{\sigma\in S_n}
\mathrm{sgn}(\sigma)
(a_0,a_{\sigma(1)},\cdots,
a_{\sigma(n)})$$
and let
$\mathrm{Tr}:\mathrm{CC}_\bullet^\lambda(M(A))\to\mathrm{CC}_\bullet^\lambda(A)$
be
$$\mathrm{Tr}(a_0,a_1,
\cdots,a_n)
=\sum_{i_0,\cdots, i_n}((a_0)_{i_0i_1},
(a_1)_{i_1i_2},\cdots, (a_n)_{i_ni_0}).
$$
\end{definition}

\begin{theorem}[{Loday-Quillen-Tsygan; see
\cite[Theorem 10.2.4]{Loday}}]
Let $A$ be an associative $k$-algebra. Then $$\mathrm{Tr}\circ\theta: \mathrm{CE}_{\bullet}(\mathfrak{gl}(A)) \to \mathrm{CC}_\bullet^\lambda(A) $$ is a chain homomorphism, which, when restricting on the primitive elements, is an isomorphism, and therefore, it induces an isomorphism $$\mathrm{H}_\bullet( \mathfrak{gl}(A))\stackrel{\cong} \to \mathbf{\Lambda}^\bullet (\mathrm{HC}_\bullet(A)[1]).$$
\end{theorem}

\subsection{The isomorphism for coalgebras}
We study the Loday-Quillen-Tsygan isomorphism for coalgebras. Our treatment follows the idea of Loday \cite[Theorem 10.2.7]{Loday} and is slightly different from Kaygun \cite{Kay}.

\subsubsection{Homology of Lie coalgebras}

For a graded vector space $L$ together with a coalgebra structure $\Delta$, if this coalgebra is not cocommutative, one may put $\delta\coloneqq \Delta - \Delta^{op}$ where $\Delta^{op}$ produce a Koszul sign.
Here $\Delta^{op}$ produces a Koszul sign. 
With $\delta$, the vector space $L$ turns out to be a Lie coalgebra, \ie $\delta$ is (graded) anti-symmetric and satisfies the co-Jacobi identity.
Since $\delta$ is anti-symmetric, we write
the image of $\delta$ as in $L\wedge L$, and
then $\delta$ 
induces a well-defined map
$$d:\mathbf\Lambda^\bullet(L)\to 
\mathbf\Lambda^\bullet(L),
l_1\wedge l_2\wedge\cdots\wedge 
l_n\mapsto\sum_{i=1}^{n}(-1)^{|l_1|+\cdots+|l_{i-1}|+|l_{i}'|-(i-1)}
l_1\wedge\cdots\wedge \delta(l_i)\wedge\cdots\wedge l_n
$$
where $l_{i}'$ is obtained from writing $\delta(l_{i}) = \sum l_{i}' \wedge l_{i}''.$
Since $\delta$ satisfies the
co-Jacobi identity, $b^2=0$.
In the following, we call such a complex
the Chevalley-Eilenberg complex for
the Lie coalgebra $L$, and denote
it by
$\mathrm{CE}_\bullet(L)$.
The associated homology is denoted by
$\mathrm H_\bullet(L)$.

\begin{example}[Matrices in a finite-dimensional coalgebra]\label{ex:MatrixcoLie}
Let $C$ be a finite dimensional coalgebra. Let $M_r(C)$ be the set of $r\times r$ matrices with entries in $C$. Let $$ \begin{array}{cccl} \Delta:&M_r(C)&\longrightarrow &M_r(C)\otimes M_r(C)\\[1mm] &E_{ij}^{a}&\longmapsto&\sum_{k}\sum_{(a)} E_{ik}^{a'}\otimes E_{kj}^{a''}, \end{array} $$ where we write the coproduct of $a$ as $\sum_{(a)}a'\otimes a''$. 
Now let $$\begin{array}{cccl} \nu: &M_r(C)&\longrightarrow &M_r(C)\otimes M_r(C)\\[1mm] &E_{ij}^a&\longmapsto& \Delta(E_{ij}^a)-\Delta^{op}(E_{ij}^a). \end{array} $$ 
Then $(M_r(C),\nu)$ is a Lie coalgebra. If $A$ is the linear dual algebra of $C$, then $M_r(A)$ together with the dual operator $\nu^*$ is exactly the matrix Lie algebra of $A$. 
In the following, we write $(M_r(C),\nu)$ as $\mathfrak{gl}_r^c(C)$.
\end{example}

\subsubsection{The isomorphism}

Let $C$ be a coalgebra and let $\mathfrak{gl}^c_r(C)$ be the Lie coalgebra given in Example \ref{ex:MatrixcoLie}. 
Let $A$ denote the dual algebra of $C$. Then $\mathfrak{gl}^c_r(C)$ is a Lie module over $\mathfrak{gl}_r(k)$, where the Lie action is given by 
\begin{equation}\label{eq:dualLieaction} (h, f)\mapsto [h,f]:=f\cdot h^T-h^T\cdot f, \quad\mbox{for all}\quad h\in \mathfrak{gl}_r(k), f\in\mathfrak{gl}_r^c(C), \end{equation} 
where $h^T$ means the transpose of $h$.

\begin{remark}
The advantage of giving
the Lie action as in \eqref{eq:dualLieaction}
is as follows:
If we view $\mathfrak{gl}_r^c(C)$
as the linear dual of $\mathfrak{gl}_r(A)$,
then \eqref{eq:dualLieaction}
is exactly the induced action
dual to the Lie action
of $\mathfrak{gl}_r(k)$
on $\mathfrak{gl}_r(A)$.
\end{remark}

There is a surjection
$\mathfrak{gl}_{r+1}^c(C)
\to \mathfrak{gl}_{r}^c(C)$
which forgets the last row and column
of the matrices.
Hence $\{\mathfrak{gl}_{r+1}^c(C)\}$
forms an inverse system.
Let $\mathfrak{gl}^c(C)
=\lim\limits_{\longleftarrow}
\mathfrak{gl}^c_r(C)$. We have
$\mathrm{CE}_\bullet(
\mathfrak{gl}^c(C))
=\lim\limits_{\longleftarrow}\mathrm{CE}_\bullet(\mathfrak{gl}_r^c(C))$.
The following lemma is parallel to the
algebra case:

\begin{lemma}
$\mathrm{CE}_\bullet(\mathfrak{gl}^c(C))
\cong \mathrm{CE}_\bullet(
\mathfrak{gl}^c(C))^{\mathfrak{gl}}\oplus 
L_*$, where $L_*$ is an
acyclic complex. 
\end{lemma}

\begin{proof}
This is a dual version
of Loday \cite[Proposition 10.1.7-8]{Loday},
where the adjoint action of 
$\mathfrak{gl}(k)$ on 
$\mathfrak{gl}(A)$ is now replaced by 
its co-adjoint action 
\eqref{eq:dualLieaction} on
$\mathfrak{gl}^c(C)$.
\end{proof}

\begin{theorem}[Loday-Quillen isomorphism for coalgebras]
\label{thm:LQTforcoalg}
Suppose $C$ 
is a coalgebra.
Then
$\mathrm{CE}_\bullet
(\mathfrak{gl}^c(C))^{\mathfrak{gl}}$ has 
a commutative and
cocommutative Hopf algebra structure,
and there exists an isomorphism
\begin{equation}
\label{iso:LodayQuillenforcoalgebras}
\Theta:
\mathrm{CC}_\bullet^{\lambda}
(C)[1]\stackrel{\cong}
\longrightarrow
\mathrm{Indec}\,\mathrm{CE}_\bullet
(\mathfrak{gl}^c
(C))^{\mathfrak{gl}(k)},
\end{equation}
which induces an isomorphism
of Hopf algebras
$\mathbf{\Lambda}^\bullet
(\mathrm{HC}_\bullet(C)[1])
\cong\mathrm{H}_\bullet(\mathfrak{gl}^c(C))$.
\end{theorem}

We start the proof with the following lemma.

\begin{lemma}\label{fixedpart}
Let $r$ and $n$ be two integers.
\begin{enumerate}[label = \rm{(}\arabic*),noitemsep,leftmargin=*,nosep]
    \item 
The dual of adjoint action 
\begin{align}
&\begin{array}{ccl}
\GL_r(k) & \longrightarrow & \Aut(\gl_r(k)^{\otimes n});\\[1mm]
g & \longmapsto & \big((f_i)_{i=1,\cdots,n}\mapsto ((g^{-1})^{T}\cdot f_i\cdot g^T))_{i=1,\cdots,n}\big)
\end{array}\label{GLaction}\\\intertext{induces }
&\begin{array}{ccl}\gl_r(k) & \longrightarrow & \End(\gl_r(k)^{\otimes n});\\[1mm] 
h & \longmapsto & \big((f_i)_{i=1,\cdots,n}\mapsto \sum_{i}(f_1,\cdots,f_i\cdot h^{T} - h^{T}\cdot f_i,\cdots,f_n)\big).\end{array}
\label{glaction}\end{align}
    \item
Assume that $k$ has characteristic $0$ and $r > 2.$ 
With the action in \text{\rm{(1)}}, we have
\[(\gl_r(k)^{\otimes n})^{\gl_r(k)} = (\gl_r(k)^{\otimes n})^{\GL_r(k)}.\]
\end{enumerate}
\end{lemma}

\begin{proof}
(1) is \cite[Theorem 10.23]{Milne}.
To show (2), consider the group $\Aut_{e}(r,k)$ of elementary automorphisms which is a subgroup of $\GL_r (k)$ consisting of elements $e^h \coloneqq \sum_{i\geq 0}h^i/i!$ for all nilpotents $h\in \gl_r(k)$.
By the Lie functor, the action of $\Aut_e(r,k)$ on $\gl_r(k)^{\otimes n}$ induces an action of $\mathrm{Lie}(\Aut_e(r,k))$ on $\End(\gl_r(k)^{\otimes n})$.
By (1), the action of $\Aut_e(r,k)$ on $\gl_r(k)^{\otimes n}$ is as in (\ref{glaction}).
Next, note that $\Aut_e(r,k)$ is a normal subgroup of $\GL_r(k)$.
As the normal subgroups of $\GL_r(k)$ not contained in the center must contains $\SL_r(k),$ we have $\Aut_e(r,k) = \SL_r(k)$. 

Finally, if $f\in \gl_{r}(k)^{\otimes n}$ is fixed by any $h$, then $f$ is also fixed by $e^{h}.$ 
Conversely assume that $f$ is fixed by all $e^{h}\in \Aut_e(r,k)$, \ie $(e^{h}-\id).f = 0,$ where $(e^{h}-\id).f$ denotes the group action on $f$.
We have $(\id+\phi)h.f = 0$ for $\phi = \sum_{i\geq 1}h^i/(i+1)!.$
As $\phi$ is nilpotent, $\id + \phi$ is invertible, the element $f$ must be fixed by $h$.
By the previous paragraph, we have the desired equation.
\end{proof}

We will need the second fundamental theorem of invariant theory (see
\cite[Theorem 9.1.3]{Loday}):

\begin{proposition}\label{secondfundthm}
Let $S_n$ act
on $V^{\otimes n}$
given by $(v_1,\ldots,v_n)\mapsto (v_{\sigma(1)},\ldots,v_{\sigma(n)})$,
for $\sigma\in S_n$.
Let $\mu$ denote the isomorphism $k[S_n] \to (\gl_r(k)^{\otimes n})^{\GL_r(k)} \cong \End(V^{\otimes n})^{\GL_r(k)}$ with $V$ the $r$ dimensional vector space.
If $\sigma = (1\cdots n),$ then
\[\mu(\sigma) = \sum_{{\bf i}}E_{i_1i_2}\otimes E_{i_2i_3}\otimes \cdots \otimes E_{i_ni_1},\]
where ${\bf i}$ denotes $n$-tuples  $(i_1,\ldots,i_n)$ and the sum extends over all $n$-tuples such that each $i_{j} = 1,\ldots,n.$ 
Here, given any basis $\{e_1,\ldots,e_r\}$ of $V$, the action of $\mu(\sigma)$ on $e_{i_1}\otimes e_{i_2}\otimes \cdots \otimes e_{i_n}$ results in $e_{i_2}\otimes \cdots \otimes e_{i_n}\otimes e_{i_1}.$ 
\end{proposition}

\begin{proof}[Proof of Theorem
\text{\rm{\ref{thm:LQTforcoalg}}}]
Assume $r>n.$
As $S_n$ is finite and $k$ has characteristic $0$, we have $V^{S_n} \cong V_{S_n}$ for any $k$-vector space $V$ with $S_n$ action. 
Hence \[(\gl^{c}(C)^{\otimes\bullet})_{S_n} \cong (\gl^{c}(C)^{\otimes\bullet})^{S_n}.\]
We see that for $r$ bigger than $n$,
\begin{align*}\begin{alignedat}{2}
\mathrm{CE}_n(\mathfrak{gl}_r^c(C))^{
\mathfrak{gl}_r(k)}&\cong \left((\gl_r^c(C)^{\otimes n})^{S_n}\right)^{
\mathfrak{gl}_r(k)}\\
&\cong \left((
\mathfrak{gl}_r(k)^{\otimes n}
\otimes C^{\otimes n})^{S_n}\right)^{
\mathfrak{gl}_r(k)}\\
&\cong \left((
\mathfrak{gl}_r(k)^{\otimes n}
\otimes C^{\otimes n})^{
\mathfrak{gl}_r(k)}\right)^{S_n}\\
&\cong \left((
\mathfrak{gl}_r(k)^{\otimes n})^{\GL_r(k)}
\otimes C^{\otimes n}\right)^{S_n} && \,\,\text{(by Lemma~\ref{fixedpart})}\\
&\cong (k[S_n]\otimes C^{\otimes n})^{S_n} && \,\,\text{(by Proposition~\ref{secondfundthm})}.
\end{alignedat}
\end{align*}
Put $N_\bullet = \bigoplus_{n\geq 0}(k[S_n]\otimes C^{\otimes n})^{S_n}.$ 
Let $\Theta^{*}$ denote the isomorphism $N_\bullet \to \CE_\bullet(\gl_r^{c}(C))^{
\mathfrak{gl}_r(k)}$.
Note $\CE_\bullet(\gl^{c}(C))^{\gl(k)} \cong (\CE^{\bullet}(\gl(A))_{\gl(k)})^*$ and $((k[S_n]\otimes C^{\otimes n})^{S_n} \cong (k[S_n]\otimes A^{\otimes n})_{S_n})^*$ for $A$ being the dual algbera of $C$.
The map $\Theta^{*}$ is the dual of the isomorphism (denoted $\Theta$) in (\ref{iso:decompCE2}). 
Moreover, a Hopf algebra structure is constructed on $L_{\bullet} \coloneqq \bigoplus_{n\geq 0}(k[S_n]\otimes A^{\otimes n})^{S_n}$ by \cite[10.2.12-10.2.17]{Loday}.
Hence, there is an induced Hopf algebra structure on $N_\bullet$. 
We are to describe the product in the Hopf algebra structure of $N_{\bullet}$ explicitly.
Notice that for any $n<r$ if $\{\sigma \otimes (c_{i})_{i=1,\cdots,n}\}$ is a basis of $k[S_n]\otimes C^{\otimes n},$ then the set 
\[\big\{\overline{\sigma\otimes (c_i)}\big\}\text{ with }\overline{\sigma\otimes (c_i)}\coloneqq\sum_{\hat{\sigma}\in S_n} \hat{\sigma}.(\sigma \otimes (c_{i})_{i=1,\ldots,n})\] is a basis of $N_n.$

\begin{lemma}\label{multiplication}
The product of $N_\bullet$ is explicitly 
\begin{eqnarray*}
m&:&\overline{\sigma\otimes (b_i)_{i=1,\ldots,m}}\otimes \overline{\varsigma\otimes (c_{i})_{i=1,\ldots,n}} \mapsto \overline{\sigma\varsigma'\otimes (b_{i})_{i=1,\ldots,m}\otimes (c_j)_{j=1,\ldots,n}}\in L_{m+n},
\end{eqnarray*}
where $\varsigma'$ satisfies $\varsigma'(i) = i$ if $i\leq m$ and $\varsigma'(i)=m+\varsigma(i-m)$ if $i>m.$
\end{lemma}
\begin{proof}
Let $\Delta$ denote the coproduct of $L_\bullet$, and $m_k$ denote multiplication of the field $k$.
For any $f,g\in N_{\bullet}$ and $a\in L_{\bullet},$ we need to check
\begin{align}
m(f\otimes g)(a) = m_{k}\big((f \otimes g)\Delta(a)\big).\label{multco}\end{align}
We are to work out $m(\overline{\sigma\otimes (b_i)_{i=1,\ldots,m}}\otimes \overline{\varsigma\otimes (c_{j})_{j=1,\ldots,n}})$ using (\ref{multco}) and show the obtained result is exactly the one in Lemma~\ref{multiplication}.
We may assume that $a$ is monomial, \ie having only a single summand.
Due to the coproduct in $L_{\bullet}$, if $m_k((f\otimes g)\Delta(a))$ is nontrivial, then $a$ must has the form $\sigma\varsigma' \otimes (b_{i}^*)\otimes (c_{j}^{*})$. 
Indeed, in $L_{m+n},$ we have $\sigma\varsigma' \otimes (b_{i}^*)\otimes (c_{j}^{*}) = \delta\sigma\varsigma'\delta^{-1} \otimes \delta((b_{i}^*)\otimes (c_{j}^{*}))$ for any $\delta\in S_{m+n}$. 
This implies that $m(f\otimes g)$ has the form as in the lemma.
\end{proof}

Let $U_n$ denotes the set of conjugacy classes of $(1,\ldots,n)$ in $S_n$.
By Lemma~\ref{multiplication}, the set of indecomposables are
$\bigoplus_{n\geq 0}(k[U_n]\otimes C^{\otimes n})^{S_{n}} = \bigoplus_{n\geq 0} (C^{\otimes n})^{\mathbb Z_n}$,
which is exactly the cyclic
complex $\mathrm{CC}_\bullet^\lambda(C)[1]$.
We need to show that the map $\Theta^*$ preserves Hopf algebra structures and differentials.
As $\Theta$ preserves the Hopf algebra structures of $\CE_{\bullet}(\gl(A))_{\gl(k)}$ and $L_\bullet$, its dual $\Theta^{*}$ preserves the dual of the Hopf algebra structures.
To check that the map $\Theta^{*}$ preserves the differential, we first work out the image of $\overline{\sigma\otimes (c_i)}\in N_n$ under $\Theta^{*}$.

The isomorphism $\mu$ in Proposition~\ref{secondfundthm} induces \[
\begin{array}{cccl}
M: & (k[S_n]\otimes C^{\otimes n})^{S_n} & \longrightarrow & \left((
\mathfrak{gl}_r(k)^{\otimes n})^{\GL_r(k)}
\otimes C^{\otimes n}\right)^{S_n}\\[1mm]
&\overline{\sigma\otimes (c_i)} & \longmapsto & \sum_{\hat{\sigma}\in S_n}\sum_{\bf{i}}\sgn(\hat{\sigma})\cdot E_{\hat{\sigma}\sigma\hat{\sigma}^{-1},{\bf i}}\otimes (c_{\hat{\sigma}(1)},c_{\hat{\sigma}(2)},\ldots,c_{\hat{\sigma}(n)}),\end{array}\]
where we write
\begin{eqnarray*}
E_{\hat{\sigma}\sigma\hat{\sigma}^{-1},{\bf i}} & = & E_{i_{1}i_{\hat{\sigma}\sigma\hat{\sigma}^{-1}(1)}}\otimes E_{i_{2}i_{\hat{\sigma}\sigma\hat{\sigma}^{-1}(2)}}\otimes \cdots \otimes E_{i_{n}i_{\hat{\sigma}\sigma\hat{\sigma}^{-1}(n)}}
\end{eqnarray*}
The sign $\sgn(\hat{\sigma})$ is the Koszul sign obtained from permuting elements in $C^{\otimes n}$ using $\hat{\sigma}$.
We may rewrite this map as
\begin{eqnarray*}
\overline{\sigma \otimes (c_i)} & \mapsto & \sum_{\hat{\sigma}\in S_n}\sum_{{\bf i}} \sgn(\hat{\sigma})\cdot E_{i_{\hat{\sigma}(1)}i_{\hat{\sigma}\sigma(1)}}\otimes E_{i_{\hat{\sigma}(2)}i_{\hat{\sigma}\sigma(2)}} \otimes \cdots \otimes E_{i_{\hat{\sigma}(n)}i_{\hat{\sigma}\sigma(n)}} \\ 
& & \hspace{7.6cm} \otimes (c_{\hat{\sigma}(1)},c_{\hat{\sigma}(2)},\ldots,c_{\hat{\sigma}(n)}).
\end{eqnarray*}
Hence we have
\begin{eqnarray}
    \overline{\sigma\otimes(a_i)} & \mapsto & \sum_{{\bf i}}E_{i_1i_{\sigma(1)}}^{c_1}\wedge E_{i_2i_{\sigma(2)}}^{c_2} \wedge \cdots \wedge E_{i_ni_{\sigma(n)}}^{c_n}\in \mathrm{CE}_n(\mathfrak{gl}_r^c(C))^{
\mathfrak{gl}_r(k)}.\label{Theta*}\end{eqnarray}
Overall, we have
\[\begin{array}{ccl}
\CC_\bullet^{\lambda}(C)[1] &\longrightarrow& \CE_{\bullet}(\gl_r^{c}(C)^{\gl(k)});\\[1mm] [c_i]_{i=1,\ldots,n} &\longmapsto &\sum_{{\bf i}}E_{i_1i_{2}}^{c_1}\wedge E_{i_2i_{3}}^{c_2} \wedge \cdots \wedge E_{i_ni_{1}}^{c_n},\end{array}\]
where $[c_i]_{i=1,\ldots,n}\in \CC_{n-1}(C)[1]$ and a square bracket is used here to denote $N((c_i)_{i=1,\ldots,n})$ with $N$ in (\ref{def:totalcomplex}). 
It remains to show that this map preserves the differential of $\CC_\bullet^{\lambda}(C)[1].$
We need to check $\Theta^{*}\circ b = b' \circ \Theta^{*}.$
As for the left side, by the commutativity of (\ref{def:totalcomplex}), we have 
\begin{align}
b([c_i]_{i=1,\ldots,n}) = N(b'(c_1\otimes[c_2,\ldots,c_n]))\label{bandbprime}\end{align} 
and hence
\begin{eqnarray*}
& & \Theta^{*}(b([c_1,c_2,
\cdots, c_n])) \\
& = & \Theta^{*}\bigg(\sum_{j=1}^{n}
\sum_{(c_j)}(-1)^{|c_1|+
\cdots+|c_{j-1}|+|c_{j}'|-(j-1)-1}[c_{1},\ldots,c_{j}',c_{j}'',
\ldots,c_{n}]\bigg) \\
& = & \sum_{{\bf i_{*}}}\sum_{j=1}^{n}\sum_{(c_j)}
(-1)^{|c_1|+\cdots+|c_{j-1}|+
|c_{j}'|-j}E_{i_1i_2}^{c_1}\wedge
E_{i_2i_3}^{c_2}
\wedge
\ldots\wedge E_{i_ji_*}^{c_{j}'}\wedge E_{i_*i_{j+1}}^{c_{j}''}\wedge\ldots\wedge E_{i_ni_1}^{c_{n}},
\end{eqnarray*}
where ${\bf i_*}$ denotes $(n+1)$-tuple $(i_1,\ldots,i_n,i_*)$ with $i_j,i_* = 1,\ldots,r.$ 
We have 
As for $b'\circ \Theta^{*}$, we have
\begin{eqnarray*}
& & d\bigg(\sum_{{\bf i}}E_{i_1i_{2}}^{c_1}\wedge E_{i_2i_{3}}^{c_2} \wedge \cdots \wedge E_{i_ni_{1}}^{c_n}\bigg)\\
&=& \sum_{{\bf i}}\sum_{j=1}^{n}\sum_{k=1}^{r}\sum_{(c_j)}(-1)^{|c_{1}|+\cdots+|c_{j-1}|+|c_{j}'|-j}E_{i_1i_{2}}^{c_1}\wedge
E_{i_2i_3}^{c_2}\wedge \cdots \wedge E_{i_jk}^{c_j'}\wedge E_{ki_{j+1}}^{c_j''} \wedge \cdots \wedge E_{i_ni_{1}}^{c_n}.
\end{eqnarray*}
By a straightforward comparison, the desired equation holds.

Since $\mathfrak{gl}_{r+1}^c(C)\to
\mathfrak{gl}_{r}^c(C)$ is surjective,
and so is the induced
homomorphism on the chain complexes.
From this we
get that
$\mathrm{CE}_\bullet(\mathfrak{gl}_{r+1}^c(C))
\to 
\mathrm{CE}_\bullet(\mathfrak{gl}_{r}^c(C))$
satisfies the Mittag-Leffler condition,
\eqref{iso:LodayQuillenforcoalgebras}
is a quasi-isomorphism
of chain complexes by taking
the inverse limit.
Thus we also get the isomorphism
of Hopf algebras
$\mathbf{\Lambda}^\bullet
(\mathrm{HC}_\bullet(C)[1])
\cong\mathrm{H}_\bullet(\mathfrak{gl}^c(C))$.
\end{proof}

\subsection{Koszul duality}

The following proposition gives
an isomorphism between the homologies
of $\mathfrak{gl}(A)$ and
$\mathfrak{gl}^c(A^{\ac})$, which
generalizes Proposition
\ref{prop:qisinKoszulCY}.

\begin{proposition}\label{prop:qismatrix}
Let $A$ be Koszul algebra.
Denote by $A^!$ and $A^{\ac}$ its
Koszul dual algebra and coalgebra
respectively. Then we have
commutative and cocommutative
Hopf algebra quasi-isomorphisms
$$
\xymatrix{
&\mathrm{CE}_\bullet(
\mathfrak{gl}(\Omega(A^{\ac}))
)_{\mathfrak{gl}(k)}\ar[rd]^{\simeq}&\\
\mathrm{CE}_\bullet(
\mathfrak{gl}(A^{\ac})
)^{\mathfrak{gl}(k)}\ar[ru]^{\simeq}
&&
\mathrm{CE}_\bullet(
\mathfrak{gl}(A)
)_{\mathfrak{gl}(k)}
}$$
\end{proposition}

\begin{proof}
Since $\mathrm{CE}_\bullet
(\mathfrak{gl}(\Omega(A^{\ac})))_{
\mathfrak{gl}
}$
is a commutative and cocommutative
Hopf algebra,
the indecomposables are identified
with the primitives. Thus combining
with 
Proposition
\ref{prop:qisinKoszulCY},
we have the following 
diagram:
$$
\xymatrixrowsep{0.5pc}
\xymatrix{
&\mathrm{Prim}\,\mathrm{CE}_\bullet(
\mathfrak{gl}(\Omega(A^{\ac}))
)_{\mathfrak{gl}(k)}\ar@{.>}[rd]
\ar[dd]^{\cong}_{\textup{Prop.
\ref{prop:primitive}}}&\\
\mathrm{Indec}\,\mathrm{CE}_\bullet(
\mathfrak{gl}(A^{\ac})
)^{\mathfrak{gl}(k)}\ar@{.>}[ru]
&&
\mathrm{Prim}\,\mathrm{CE}_\bullet(
\mathfrak{gl}(A)
)_{\mathfrak{gl}(k)}
\ar[dd]^{\cong}_{\textup{Prop.
\ref{prop:primitive}}}\\
&\mathrm{CC}_\bullet^\lambda
(\Omega(A^{\ac}))\ar[rd]^{\simeq}
&\\
\mathrm{CC}^\lambda_\bullet(A^{\ac})
\ar[ur]^{\simeq}\ar[uu]_{\cong}^{
\textup{Thm. \ref{thm:LQTforcoalg}}
}
&&
\mathrm{CC}^\lambda_\bullet(A),
}$$
where the map
$\mathrm{Indec}\,\mathrm{CE}_\bullet(
\mathfrak{gl}(A)
)^{\mathfrak{gl}(k)}
\to
\mathrm{Prim}\,\mathrm{CE}_\bullet(
\mathfrak{gl}(\Omega(A^{\ac}))
)_{\mathfrak{gl}(k)}
$
is given as follows:
note that the
vertical morphisms in the left
square are isomorphisms
by Theorem 
\ref{thm:LQTforcoalg}
and Proposition
\ref{prop:primitive}, and thus
by the same method
as in Proposition
\ref{prop:qisforKoszul},
we can similarly
get the quasi-isomophism.
The quasi-isomorphsm
$\mathrm{Prim}\,\mathrm{CE}_\bullet(
\mathfrak{gl}(\Omega(A^{\ac}))
)_{\mathfrak{gl}(k)}\to
\mathrm{Prim}\,\mathrm{CE}_\bullet(
\mathfrak{gl}(A)
)_{\mathfrak{gl}(k)}
$
is constructed analogously.
Taking the corresponding
Hopf algebras, we get the desired
quasi-isomorphism.
\end{proof}

Since $$\mathrm{CE}_\bullet(
\mathfrak{gl}(A)_{\mathfrak{gl}(k)}
)
\stackrel{\simeq}\hookrightarrow
\mathrm{CE}_\bullet(
\mathfrak{gl}(A)_{\mathfrak{gl}(k)}
)\oplus L_*
\cong
\mathrm{CE}_\bullet(
\mathfrak{gl}(A)
)$$
and
$$\mathrm{CE}_\bullet(
\mathfrak{gl}^c(A^{\ac})^{\mathfrak{gl}(k)}
)
\stackrel{\simeq}\hookrightarrow
\mathrm{CE}_\bullet(
\mathfrak{gl}^c(A^{\ac})^{\mathfrak{gl}(k)}
)\oplus L_*
\cong
\mathrm{CE}_\bullet(
\mathfrak{gl}^c(A^{\ac})
)$$
(both $L_*$'s are acyclic),
the above proposition then
implies the following.

\begin{theorem}
\label{cor:isoHmlgyofLiecylic}
Let $A$ be Koszul algebra.
Denote by $A^{\ac}$ its
Koszul dual coalgebra.
Then we have the following
commutative diagram of Hopf
algebras:
\begin{equation}
\label{iso:squareofLieinCY}
\begin{split}
\xymatrixcolsep{3pc}
\xymatrix{
\mathrm{H}_\bullet(
\mathfrak{gl}^c(A^{\ac}))
\ar[r]^-{\textup{Koszul}}_{\cong}
&
\mathrm{H}_\bullet(
\mathfrak{gl}
(A))\ar[d]^{\textup{LQT}}_{\cong}
\\
\mathbf{\Lambda}^\bullet
(\mathrm{HC}_\bullet(A^{\ac})[1])
\ar[u]^{\textup{LQT}}_{\cong}
\ar[r]^-{\textup{Koszul}}_{\cong}&
\mathbf{\Lambda}^\bullet
(\mathrm{HC}_\bullet(A)[1])
}
\end{split}
\end{equation}
\end{theorem}

\section{Algebraic 
K-theory}\label{sect:Ktheory}

In a recent paper \cite{Hennion},
Hennion proved that
the cyclic homology
of a (DG) algebra is isomorphic
to the tangent complex of its K-theory.
Prior to Hennion, Bloch \cite{Blo}
defined
a version of tangent complex to K-theory and Goodwillie \cite{Goo} continued to
show that the relative tangent complex is isomorphic
to the relative cyclic homology.
The novelty of Hennion's result
is that, with his new definition
of the tangent complex, the tangent map 
of the morphism $\mathrm{BGL}_\infty\to
\mathrm K$ is exactly the Loday-Quillen-Tsygan isomoprhism.

\subsection{K-theory of DG algebras}
\label{subsect:Kthyofalgebra}

In this subsection, we study
the tangent complex of the K-theory
for DG algebras. 
We present both Quillen's approach
and Waldhausen's approach,
which will be used later for the
comparision of the
K-theory of DG coalgebras.


Let us first recall that
the category of (non-negatively graded) 
DG algebras
is equivalent to the categogry
of simplicial algebras.

Let $A$ be a simplicial algebra over $k$.
Denote by $\mathrm M_n(A)$ the simpicial set obtained
by taking $n\times n$-matrices level-wise.
Let $\mathrm{GL}_n(A)$ be the group of invertible
matrices as the pullback
$$
\xymatrix{
\mathrm{GL}_n(A)\ar[r]\ar[d]&\mathrm M_n(A)\ar[d]\\
\mathrm{GL}_n(\pi_0A)\ar[r]&\mathrm M_n(\pi_0A)
}$$
Denote by $\mathrm{BGL}_n(A)$
its classifying space
and by $\mathrm{BGL}_{\infty}(A)$
the colimit $\mathrm{colim}_n\mathrm{BGL}_n(A)$.
Applying Quillen's +construction
to $\mathrm{BGL}_{\infty}(A)$
gives a model for the K-theory
of $A$, which is in fact an
infinite loop space. We have 
an equivalence
\begin{equation}
\label{const:KofQuillen}
\mathrm K_0\times\mathrm{BGL}_\infty(-)^{+}
\simeq \Omega^{\infty}K.
\end{equation}

\subsection{The tangent complex of K-theory}\label{subsect:Hennion}

In this subsection, we recall
Hennion's result on the tangent
complex of K-theory. More details
can be found in \cite{Hennion}.

In the following, denote by
$\mathbf{dgArt}_k$ the category of 
dg-Artinian commutative $k$-algebra
with residue field $k$,
by $\mathbf{Sp}_{\ge 0}$ 
the $\infty$-category of connective spectra,
and by $\mathbf{sSet}$
be the $\infty$-category of spaces.

Fix $A$ to be a unital simplicial $k$-algebra.
Consider the two following functors
\begin{equation}\label{functor:BGL}
\begin{split}
\begin{array}{cccl}
\overline{\mathrm{BGL}}_{\infty}
(\mathscr{A}_A)
:&\mathbf{dgArt}_k
&\longrightarrow&\mathbf{sSet}  \\
&B&\longmapsto&\mathrm{hofib}
(\mathrm{BGL}_\infty(A\otimes_k B)
\to\mathrm{BGL}_\infty(A))
\end{array}
\end{split}
\end{equation}
and
\begin{equation}\label{functor:K}
\begin{split}
\begin{array}{cccl}
\overline{\mathrm K}
(\mathscr{A}_A):&\mathbf{dgArt}_k
&\longrightarrow&\mathbf{Sp}_{\ge 0}  \\
&B&\longmapsto&\mathrm{hofib}
(\mathrm K(A\otimes_k B)
\to\mathrm K(A)).
\end{array}
\end{split}
\end{equation}
Hennion showed that
the tangent complexes
to these two functors
are quasi-isomorphic to $\mathrm{HC}_\bullet(A)$ and 
$\mathfrak{gl}_\infty(A)$
respectively,
where the Lie algebra structure
on the former is the usual matrix
Lie algebra while
on the latter is abelian 
(see \cite[Corollaries 3.1.3\&4.1.5]{Hennion}).
We next explain 
what he means by the tangent complex of
these two functors.

\subsubsection{Formal moduli problem
and its tangent complex}

Let us start with the
notion of a formal moduli problem.

\begin{definition}Let $\mathscr C$ be an $\infty$-category
with finite limits.
A functor $F:\mathbf{dgArt}_k\to\mathscr C$
is said to satisfy the condition (S)
if
$$
(S)\quad 
\left\{\begin{array}{cl}
(S1)&\mbox{The object $F(k)$ is final
in $\mathscr C$.}\\
(S2)&\mbox{For any map $A\to B
\in\mathbf{dgArt}_k$ that is injective
on $\mathrm H^0$, the induced}\\
&\mbox{morphism $F(k\times_B A)
\to F(k)\times_{F(B)}F(A)$ is an equivalence.}
\end{array}
\right.$$
\end{definition}

\begin{definition}
A {\it pre-formal moduli problem} (pre-FMP)
is a functor
$\mathbf{dgArt}_k\to\mathbf{sSets}$.
Denote by $\mathbf{PFMP}_k$ the
category of pre-FMPs.
Then A {\it formal moduli problem} 
(FMP)
is a pre-FMP $F$ that satisfies
the condition $(S)$. Denote its 
category
by $\mathbf{FMP}_k$.
\end{definition}

Let $\mathrm{i}: \mathbf{FMP}_k
\to\mathbf{PFMP}_k$ be the inclusion
functor. 
The category 
$\mathbf{PFMP}_k$
is representable, and the full
subcategory $\mathbf{FMP}_k$
is strongly reflexive. In particular,
the inclusion $\mathrm i:
\mathbf{FMP}_k\to\mathbf{PFMP}_k$
admits a left adjoint, which
is called the {\it formalization 
functor}
and is denoted by $\mathrm L$.

Given formal moduli problem, one
may define its tangent complex,
which is a functor
$\mathbb T:\mathbf{FMP}_k
\to\mathbf{dgMod}_k$
(see \cite[Definition 1.2.5]{Lurie}).
To this end, let us first recall:

\begin{definition}
Suppose $V$ is a complex over $k$.
A {\it shifted dg-Lie algebra} on $V$
is a dg-Lie algebra structure
on $V[-1]$. Denote
by $\mathbf{dgLie}_k^\Omega$
the $\infty$-category of shifted
dg-Lie algebras.
\end{definition}

Let $\mathbf{dgMod}_k$ be the
$\infty$-category of chain complexes
over $k$. Then there is 
natural forgetful functor
$\mathbf{dgLie}_k^\Omega\to
\mathbf{dgMod}_k$.
Then one of the main results
of Lurie and Pridham
obtained in \cite{Lurie,Pridham}
is the following 
(c.f. \cite[Theorem 2.1.7]{Hennion}):

\begin{theorem}[Pridham, Lurie]
\label{thm:PridhamLurie}
The tangent complex functor
$\mathbb T:\mathbf{FMP}_k
\to\mathbf{dgMod}_k$
factors through the forgetful functor
$\mathbf{dgLie}_k^\Omega\to
\mathbf{dgMod}_k$.
In other words, the tangent
complex $\mathbb TF$ of a formal
moduli problem admits a natural
shifted Lie algebra structure.
Moreover, the functor
$F\mapsto \mathbb TF$
induces an equivalence
$\mathbf{FMP}_k
\simeq\mathbf{dgLie}_k^\Omega$.
\end{theorem}

One can also define
the tangent complex functor
for $\mathbf{PFMP}_k$
by composing $\mathbb T$
with $L$ given above. More
precisely, it is given as follows:

\begin{definition}[Hennion]
The {\it tangent complex functor}
of $\mathbf{PFMP}_k$
is the functor
$\ell:=\mathbb T\circ \mathrm L:
\mathbf{PFMP}_k\to\mathbf{dgLie}_k^\Omega$.
\end{definition}

\subsubsection{Tangent complex
of BGL}

Given an associative algebra $A$
over $k$. For 
$n\in\mathbb N\cup\{\infty\}$,
the functor
$$\overline{\mathrm{BGL}}_n:
\mathbf{dgArt}_k\to\mathbf{sSet},
B\mapsto
\mathrm{hofib}(
\mathrm{BGL}_n(A\otimes_k B)
\to\mathrm{BGL}_n(A)
)$$
is a pre-FMP (recall that
when $n=\infty$,
it is the above \eqref{functor:BGL}).
Then Hennion proved the following:

\begin{proposition}
[{\cite[Corollary 4.1.5]{Hennion}}]
For any $n\in\mathbb N\cup\{\infty\}$,
the (shifted) tangent Lie algebra
$\ell(\overline{\mathrm{BGL}}_n(\mathcal A_A))$
is isomorphic to
$\mathfrak{gl}_n(A)[1]$.
\end{proposition}

\subsubsection{Tangent complex of K}

The tangent complex of K-theory is 
more elaborate.

Let $\mathbf{FMP}_k^{\mathbf{Ab}}$
and $\mathbf{PFMP}_k^{\mathbf{Ab}}$
denote the categor of abelian group
objects in the formal 
and pre-formal moduli problem 
respectively.
Also let
$\mathbf{dgLie}_k^{\Omega,\mathbf{Ab}}$
be the abelian group
objects in $\mathbf{dgLie}_k^{\Omega}$.
Then we have the following.

\begin{proposition}
\begin{enumerate}
\item[$(1)$]
The forgetful functor
$\mathbf{dgLie}_k^{\Omega,
\mathrm{Ab}}\to\mathbf{dgMod}_k$
is an equivalence.

\item[$(2)$] The category
$\mathbf{FMP}_k^{\mathbf{Ab}}$ is equivalent
to the category $\mathbf{dgLie}_k^{\Omega,
\mathbf{Ab}}$.
\end{enumerate}
\end{proposition}

The above 
(2) is proved in
\cite[Proposition 2.2.2]{Hennion}
and then the above
(1) is a direct corollary
of Pridham and Lurie's result 
(Theorem \ref{thm:PridhamLurie}).

Let us denote by 
$\mathrm i_{\mathrm{Ab}}:
\mathbf{FMP}_k^{\mathbf{Ab}}\to
\mathbf{PFMP}_k^{\mathbf{Ab}}$ 
the inclusion functor, and by
$\mathrm L^{\mathrm{Ab}}$ its left adjoint.

\begin{definition}Let
$\mathbb T^{\mathrm{Ab}}:
\mathbf{FMP}_k^{\mathbf{Ab}}\stackrel{\sim}
\longrightarrow
\mathbf{dgLie}_k^{\Omega,\mathrm{Ab}}
\stackrel{\simeq}\longrightarrow
\mathbf{dgMod}_k$ be
the equivalence of the above proposition.
Let $\ell^{\mathrm{Ab}}:=\mathbb 
T^{\mathrm{Ab}}\circ 
\mathrm{L}^{\mathrm{Ab}}:
\mathbf{PFMP}_k^{\mathbf{Ab}}
\to\mathbf{dgMod}_k$.
\end{definition}

Now let us move back to K-theory.
Fix an associative algebra $A$
over $k$. Consider the functor
$$\overline{\mathrm{K}}:
\mathbf{dgArt}_k\to
\mathbf{Sp}_{\ge 0}$$
given by \eqref{functor:K}.
The tangent complex
$\mathbb T_{\mathrm{K}(A),0}:=
\ell^{\mathrm{Ab}}
(\overline{\mathrm K}
(\mathscr A_A))$
is called the tangent complex
to the
K-theory of $A$.
Hennion then showed the following.

\begin{proposition}
[{\cite[Corollary 3.1.3]{Hennion}}]
The tangent complex
$\mathbb T_{\mathrm K(A),0}$
is quasi-isomorphic to 
$\mathrm{CC}_\bullet(A)[1]$.
\end{proposition}

\subsubsection{Connecting
BGL with K}

In the following, denote
by $\theta:
\mathbf{dgMod}_k\simeq
\mathbf{dgLie}_k^{\Omega,Ab}
\to\mathbf{dgLie}_k^\Omega$
the functor which
forgets the abelian group structure.
Hennion showed the following:

\begin{theorem}[{\cite[Theorem
4.2.1]{Hennion}}]
\label{thm:Hennion}
The natural map
$\overline{\mathrm{BGL}}_\infty(\mathscr A_A)
\to
\Omega^\infty\overline{\mathrm K}(\mathscr A_A)$
induces a morphism
on the tangent Lie algebras
$$T:\mathfrak{gl}_\infty(A)[1]\to
\theta(\mathrm{CC}_\bullet(A)[1]).$$
Moreover,
$T$ is homotopic to the
Loday-Quillen-Tsygam map trace map
$\mathrm{CE}_\bullet
(\mathfrak{gl}_\infty(A))
\to \mathrm{CC}_\bullet(A)[1]$.
\end{theorem}

Now let us go back
to the Koszul algebra
case. We have the following.

\begin{corollary}
Let $A$ be a Koszul algebra.
Denote by
$A^{\ac}$ its Koszul dual coalgebra.
We have the following commutative
diagram:
$$
\xymatrixcolsep{1.5pc}
\xymatrixrowsep{1pc}
\xymatrix{
&\mathrm{CE}_\bullet(\mathfrak{gl}_\infty(\Omega(A^{\ac})))\ar[rr]^{\simeq}
\ar'[d]^-{\mathrm{LQT}}[dd]
&&
\mathrm{CE}_\bullet(\mathfrak{gl}_\infty(A))\ar[dd]^{\mathrm{LQT}}\\
\mathrm{BGL}_\infty(\Omega(A^{\ac}))
\ar[rr]^{\simeq\quad\quad}\ar[dd]
\ar@{~>}[ru]
&&\mathrm{BGL}_\infty(A) 
\ar@{~>}[ru]\ar[dd]
&\\
&\mathrm{CC}_\bullet(\Omega(A^{\ac}))[1]
\ar'[r]^-{\simeq}[rr]
&&
\mathrm{CC}_\bullet(A)[1]\\
\mathrm{K}(\Omega(A^{\ac}))\ar[rr]^{\simeq}
\ar@{~>}[ru]
&&
\mathrm{K}(A)
\ar@{~>}[ru]
&
}
$$
where the horizotal
maps are all quasi-equivalences/isomorphisms,
and the slant maps
(the curved arrows) are the tangent
functors.
\end{corollary}

\begin{proof}
First, both the front and back
squares commute,
since $\Omega(A^{\ac})\to A$
is a quasi-isomorphism.
Second, the top and bottom
squares commute,
by the functoriality of
the tangent map.
Finally, the left and right
squares commute, due to 
Hennion's Theorem
\ref{thm:Hennion}.
\end{proof}

It would also
be interesting
to define a version of $\mathrm{BGL}$
and K-theory
for DG coalgebras, as well as their
tangent complexes. In particular,
the diagram
\eqref{qis:Koszuldualityforalg}
in Proposition \ref{prop:qisinKoszulCY}
can be interpreted as 
the quasi-isomorphisms 
of the following
equivalences of K-theories
$$
\xymatrixcolsep{3pc}
\xymatrixrowsep{0.5pc}
\xymatrix{
&\mathrm{K}
(\Omega(A^{\ac}))\ar[rd]^{\simeq}
&\\
\mathrm K(A^{\ac})
\ar[ur]^{\simeq}
&&
\mathrm K(A),
}
$$
by taking the tangent complexes.
In other words,
the following diagram commutes:
$$
\xymatrixcolsep{3pc}
\xymatrixrowsep{0.5pc}
\xymatrix{
&\mathrm{K}
(\Omega(A^{\ac}))\ar[rd]^{\simeq}
\ar@{~>}[dd]
&\\
\mathrm K(A^{\ac})\ar@{~>}[dd]
\ar[ur]^{\simeq}
&&
\mathrm K(A)\ar@{~>}[dd]\\
&\mathrm{CC}_\bullet^\lambda
(\Omega(A^{\ac}))\ar[rd]^{\simeq}
&\\
\mathrm{CC}^\lambda_\bullet(A^{\ac})
\ar[ur]^{\simeq}
&&
\mathrm{CC}^\lambda_\bullet(A).
}
$$
We hope to
address the problem in the future.

\section{Deformation of 
the Chevelley-Eilenberg complex}
\label{sect:deformation}

From now on we study the
deformation quantization problem of the Loday-Quillen-Tsygan
isomorphism for Koszul Calabi-Yau algebras.
In this section, we deformation
the abelian bracket and cobracket
on the primitive elements
of $\mathrm{CE}_\bullet(\mathfrak{gl}(A))$,
or equivalently, by
Theorem \ref{cor:isoHmlgyofLiecylic},
on 
$\mathrm{CC}_\bullet(A^{\ac})[1]$.

In the rest of the paper, for a cyclic chain of the form $N(c_0,c_1,\cdots, c_n)[1]\in\mathrm{CC}_\bullet^{\lambda}(A^{\ac})[1]$,
we usually write it as $[c_0,c_1,\cdots, c_n]$.

\subsection{Lie bialgebra structure on the cyclic complex}\label{sect:constofLiebialg}

In this subsection,
we first recall the Lie bialgebra
structure on 
$\mathrm{CC}_\bullet(A^{\ac})[1]$.
This is inspired by String Topology
(c.f. \cite{ChSu2}),
and has also appeared in \cite{CEG}.

Recall that a graded Lie bialgebra
structure on a graded $k$-module
$L$ if $L$ is equipped with
a graded Lie bracket $\{-,-\}: L\times L
\to L$ and a graded Lie cobracket
$\nu: L\to L\otimes L$
satisfying the following
cocycle condition
$$\nu\circ\{a,b\}=\sum_{(a)}a'\otimes
\{a'',b\}
+\sum_{(b)}(-1)^{|a|\cdot|b'|}b'\otimes\{a,b''\},$$
for all $a, b\in L$.
Here we write $\nu(a)=\sum_{(a)}a'\otimes a''$,
for any $a\in L$.
A Lie bialgebra is said to be involutive if $\{-,-\}\circ \nu$ vanishes identically.
In the definition, if 
the bracket
and cobracket
both have a degree, say $m$
and $n$ respectively,
then we say $L$
is a Lie bialgebra of
degree $(m,n)$.

Now suppose $A^{\ac}$ is a co-Frobenius 
coalgebra. We next introduce the following 
two operators 
\begin{align*}
\begin{array}{ccrcl}
\{-,-\}&\hspace{-0.2cm}:&\mathrm{CC}_\bullet
^{\lambda}
(A^{\ac})[1]
\times \mathrm{CC}_\bullet
^{\lambda}
(A^{\ac})[1]
 \longrightarrow  \mathrm{CC}_\bullet
^{\lambda}
(A^{\ac})[1],\\[1mm]
\delta&\hspace{-0.2cm}:&\mathrm{CC}_\bullet
^{\lambda}
(A^{\ac})[1] 
\longrightarrow  \mathrm{CC}_\bullet
^{\lambda}
(A^{\ac})[1]
\otimes \mathrm{CC}_\bullet
^{\lambda}
(A^{\ac})[1],\end{array}
\end{align*}
given by the following formulas:
for $\alpha=[a_1,\cdots,a_n]$, 
$\beta=[b_1,\cdots,b_m]$,
\begin{eqnarray}
\{\alpha,\beta\}&:=&\sum_{i,j}
(-1)^{\nu}
\langle a_i, b_j\rangle
[b_{1},\cdots, b_{j-1}, a_{i+1},\cdots,
a_n,a_1,\cdots, a_{i-1}, b_{j+1},\cdots,
b_m],\label{eq:bracket}\\
\delta(\alpha)&:=&\sum_{i<j}\langle a_i,
a_j\rangle
\big((-1)^{\eta}\cdot
[a_{i+1},\cdots, 
a_{j-1}]\otimes[a_{j+1},\cdots,a_n,a_1,
\cdots,a_{i-1}]\nonumber\\
&&\quad\quad\quad\quad\quad
-(-1)^{\zeta}\cdot[
a_{j+1},\cdots,a_n,a_1,\cdots,
a_{i-1}]\otimes[a_{i+1},\cdots, 
a_{j-1}]\big),\label{eq:cobracket}
\end{eqnarray}
where $(-1)^{\nu}$,
$(-1)^{\eta}$ and $(-1)^{\zeta}$
are the signs following the Koszul
sign rule (see the paragraph after
Definition \ref{def:Hochschild}).
We then have the following result from 
\cite[Theorem~9]{CEG}, which is inspired by 
the necklace Lie bialgebra 
on the cyclic paths of doubled quivers (\cite{Schedler}):  

\begin{proposition}
[{See \cite[Theorem 9]{CEG}}]
\label{prop:Liebialg}
Let $A^{\ac}$ be
a co-Frobenius coalgebra.
Then 
$\mathrm{CC}_\bullet(A^{\ac})[1]$
together with  
$\{-,-\}$ and $\delta$ given by 
\eqref{eq:bracket} and 
\eqref{eq:cobracket} is an involutative DG 
Lie bialgebra of degree $(2-n, 2-n)$.
\end{proposition}

\begin{proof}[Sketch of proof]
It is straightforward to check Lie bialgebra structure, their compatibility and involutivity.
The key point here is to check that they are compatible with the boundary map, which holds due to the co-Frobenius coalgebra structure on $A^{\ac}$ (see \eqref{eq:coproduct} in \S\ref{subsubsect:PD}).
\end{proof}

\subsection{The co-Poisson 
bialgebra structure}

Since
$\mathrm{CE}_\bullet(
\mathfrak{gl}^c(A^{\ac})
)^{\mathfrak{gl}(k)}$
is freely generated by
$\mathrm{CC}_\bullet(A^{\ac})[1]$,
by Proposition
\ref{prop:Liebialg},
the bracket
$\{-,-\}$
extends to 
$\mathrm{CE}_\bullet(
\mathfrak{gl}^c(A^{\ac})
)^{\mathfrak{gl}(k)}$
by derivation, making
it be a Poisson algebra;
similarly,
the cobracket $\delta$ 
extends to
$\mathrm{CE}_\bullet(
\mathfrak{gl}^c(A^{\ac})
)^{\mathfrak{gl}(k)}$
by co-derivation,
making it be a {\it co-Poisson coalgebra}.
Recall that a co-Poisson coalgebra
is a vector space $V$ together
with a cocommutative co-product
$\Delta: V\to V\otimes V$
and a Lie co-bracket
$\delta: V\to V\otimes V$
such that
$$(\id\otimes\Delta)\circ\delta
=(\delta\otimes \id+(\mathrm{Perm}\otimes \id)
\circ(\id\otimes \delta))\circ\Delta.$$
We next combine
these two structures together.
Let us first recall the following
notion, due to Turaev \cite{Turaev}:

\begin{definition}
\label{def:coPoissonbialg}

A {\it co-Poisson $k$-bialgebra}
is a $k$-module $A$ equipped with the 
structure of a bialgebra
and a co-Poisson coalgebra 
with the same co-multiplication $\Delta:
A\to A\otimes A$
such that
$\Delta$
and the Lie cobracket $\nu: A\to
A\otimes A$ satisfies the identity
\begin{equation}\label{eq:coPoisson}
\nu(ab)=\nu(a)\Delta(b)+\Delta(a)\nu(b), 
\end{equation}
for all $a, b\in A$.
\end{definition}

The following 
is proved by Turaev in \cite[Theorem 7.4]{Turaev}:

\begin{theorem}
[Turaev]\label{thm:coPoissonbialg}
Let $(\mathfrak{g},\{-,-\},\delta)$ 
be a Lie bialgebra.
Consider the tensor algebra
generated by $\mathfrak{g}$
over $k[h,\hbar]$, where $h$ and $\hbar$ are formal parameters.
Let $V_{h\hbar}(\mathfrak{g})$
be the quotient
algebra of this tensor algebra
by the two-sided ideal
generated by
$a\otimes b-b\otimes a
-h\{a,b\}$, which is a
bialgebra.
Let $\nu=\hbar\cdot\delta:
\mathfrak g\to 
(\mathfrak g\otimes\mathfrak g)[h,\hbar]
\subset
V_{h\hbar}(\mathfrak g)\otimes_{k[h,\hbar]}
V_{h\hbar}(\mathfrak g)$
and
extend it to
a Lie cobracket on $V_{h\hbar}(\mathfrak{g})$.
Then $V_{h\hbar}(\mathfrak{g})$
is a co-Poisson bialgebra.
\end{theorem}

Turaev \cite{Turaev} considered the
tensor algebra generated
by $\mathfrak{g}$ over
$k[h]$ (without $\hbar$),
and let $V_{h}(\mathfrak g)$
be its quotient algebra
by the same ideal as in the above
theorem. He then 
let $\nu$ be $\delta$.
Here we added an extra
formal parameter $\hbar$, since
in \S\ref{sect:quantization},
when we quantize
the Lie bialgebra,
in both
the Lie bracket 
and Lie cobracket directions,
we have non-trivial higher
order terms.
Since the theorem
is a little different
from Turaev's, we
give a proof
for completeness.

\begin{proof}
We only need to check that
\eqref{eq:coPoisson}
holds for generators 
$a, b\in\mathfrak
{g}$.
For such $a, b$,
suppose
$$\delta(a)=\sum_{(a)}
\tilde a'\otimes\tilde a''
\quad\mbox{
and}\quad
\delta(b)=\sum_{(b)}
\tilde b'\otimes\tilde b''.$$
We also have $\Delta(a)=a\otimes 1+1\otimes a$ and $\Delta(b)=b\otimes 1+1\otimes b$,
and thus
\begin{align*}
\nu(ab)-\nu(ba)&
=(\nu(a)\Delta(b)+\Delta(a)\nu(b))-
(\nu(b)\Delta(a)+\Delta(b)\nu(a))\\
&=\left(\Big(\hbar
\sum_{(a)}
\tilde a'\otimes \tilde a''
\Big)(b\otimes 1+1\otimes b)+
(a\otimes 1+1\otimes a)
\Big(\hbar
\sum_{(b)}
\tilde b'\otimes \tilde b''
\Big)
\right)
\\
&\quad-
\left(\Big(
\hbar\sum_{(b)}
\tilde b'\otimes \tilde b''
\Big)(a\otimes 1+1\otimes a)+
(b\otimes 1+1\otimes b)
\Big(\hbar
\sum_{(a)}
\tilde a'\otimes \tilde a''
\Big)
\right)\\
&=
\hbar
\sum_{(a)}\Big((\tilde a' b-b\tilde a')\otimes
\tilde a''+
\tilde a'\otimes(
\tilde a''b-b\tilde a'')
\Big)\\
&\quad+
\hbar
\sum_{(b)}\Big((a\tilde b'-\tilde b'a)\otimes
\tilde b''+
\tilde b'\otimes(
a\tilde b''-\tilde b''a)
\Big)\\
&=h\hbar
\sum_{(a)}\Big(\{\tilde a',b\}\otimes
\tilde a''
+\tilde a'\otimes\{\tilde a'',b\}
+h\hbar
\sum_{(b)}\Big(\{a,\tilde b'\}\otimes
\tilde b''
+\tilde b'\otimes\{a,\tilde b''\}\Big)\\
&=h\hbar
\cdot
\delta\circ\{a,b\}\quad\quad\mbox{(by
the cycle condition)}\\
&
=\nu(h\{a,b\}).
\end{align*}
This means $\nu$ is 
well-defined,
and the theorem follows.
\end{proof}

The above theorem also holds
for graded Lie bialgebras.
Combine the above theorem 
with Proposition
\ref{prop:Liebialg}, we have
the following.

\begin{proposition}
\label{prop:Poissonbialg}
Let $A^{\ac}$ be a (DG) co-Frobenius coalgebra.
Then with the Lie bracket
and co-bracket 
induced from
\eqref{eq:bracket} and 
\eqref{eq:cobracket}
respectively,
$V_{h\hbar}
(\mathrm{CC}_\bullet(A^{\ac})[1])$
is a co-Poisson bialgebra.
\end{proposition}

\subsection{Koszul duality
and the Loday-Quillen-Tsygan
isomorphism}

Now let us go back to the 
Calabi-Yau algebra case.
By Proposition 
\ref{prop:Poissonbialg},
the algebras in the
commutative diagram
\eqref{iso:squareofLieinCY}
in 
Theorem \ref{cor:isoHmlgyofLiecylic}
can all be deformed 
into a co-Poisson coalgebra.
Let us denote by
$
\mathrm{DH}_\bullet(\mathfrak{gl}^c(A^{\ac}))
$
and
$
\mathrm{DH}_\bullet(\mathfrak{gl}(A))
$
the induced deformation of 
$
\mathrm{H}_\bullet(\mathfrak{gl}^c(A^{\ac}))
$
and
$
\mathrm{H}_\bullet(\mathfrak{gl}(A))
$ 
respectively (here ``D" means
``deformed"). We thus get
the following.

\begin{theorem}
\label{cor:isoHmlgyofLiecylicdeformed}
Let $A$ be a Koszul Calabi-Yau
algebra.
Denote by $A^{\ac}$ its
Koszul dual
co-Frobenius coalgebra.
Then the following
\begin{equation}
\label{iso:squareofLieinCYdeformed}
\begin{split}
\xymatrixcolsep{4pc}
\xymatrix{
\mathrm{DH}_\bullet(
\mathfrak{gl}^c(A^{\ac}))
\ar[r]^-{\textup{Koszul}}_{\cong}
&
\mathrm{DH}_\bullet(
\mathfrak{gl}
(A))\ar[d]^{\textup{LQT}}_{\cong}
\\
V_{h\hbar}
(\mathrm{HC}_\bullet(A^{\ac})[1])
\ar[u]^{\textup{LQT}}_{\cong}
\ar[r]^-{\textup{Koszul}}_{\cong}&
V_{h\hbar}
(\mathrm{HC}_\bullet(A)[1])
}
\end{split}
\end{equation}
is a 
commutative diagram
of isomorphisms of
co-Poisson bialgebras,
which deforms the commutative
diagram
\eqref{iso:squareofLieinCY}.
\end{theorem}
\section{Quantization}
\label{sect:quantization}

In this section, we 
continue to study
the commutative diagram
\eqref{iso:squareofLieinCYdeformed},
and construct
a quantization 
of $V_{h\hbar}
(\mathrm{HC}_\bullet(A^{\ac})[1])$,
which then gives
a quantization
of the deformed Loday-Quillen-Tsygan
isomorphism.
As we have said before,
there is an overlap with
\cite{CEG}; since some details
in loc. cit. is omitted,
we here give a rather detailed
proof for completeness.

In what follows, by quantization of
$V_{h\hbar}
(\mathrm{HC}_\bullet(A^{\ac})[1])$,
we mean a deformation of 
its product
in the direction
of the Poisson bracket
and a deformation 
of its coproduct in
the direct of the
co-Poisson co-bracket,
with some compatibilities.
It is equivalent to 
the quantization
of the Lie bialgebra
on $\mathrm{HC}_\bullet(A^{\ac})[1]$
in the following sense.

\begin{definition}\label{d31}
Suppose $(\mathfrak g, \{-,-\},\delta)$ is a Lie bialgebra.
A {\it quantization} of $\mathfrak g$ is a Hopf algebra $A$, flat over
$k[h,\hbar]$, such that
\begin{enumerate}
\item[$(1)$]
$A/(h\cdot A+\hbar\cdot A)
\cong \Lambda^\bullet(\mathfrak g)$
as $k$-vector spaces;

\item[$(2)$] for any lifting $\tilde x,\tilde y\in A$ of $x, y\in\mathfrak g$ respectively,
$$\frac{\tilde x\tilde y-
\tilde y\tilde x}{h} = \{x,y\}
\mod \{h,\hbar\},
\quad
\frac{\Delta(\tilde x)-\Delta^{op}(\tilde x)}{\hbar}=\delta (x)\mod\{h,\hbar\}.
$$
\end{enumerate}
\end{definition}

This definition can be generalized
to DG Lie bialgebras
of degree $(m, n)$,
if we assign the gradings of $h$
and $\hbar$ to be $-m$ and $-n$
respectively.
The main result of this section is 
the following.

\begin{theorem}\label{theoremsec6}
Let $A^{\ac}$ be a co-Frobenius 
coalgebra. Then there is a
DG Hopf algebra $\mathfrak{A}$, 
constructed in \text{\rm{\S\ref{subsect:construction}}},
that quantizes the $V_{h\hbar}(\mathrm{HC}_\bullet(A^{\ac})[1])$, or equivalently, 
quantizes the Lie bialgebra $\mathrm{CC}_\bullet^{\lambda}(A^{\ac})[1]$.
\end{theorem} 

The rest of this section is as follows.
In \S\ref{subsect:construction}
we construct $\mathfrak A$ in the
above theorem;
in \S\ref{subsect:Hopf}
we show $\mathfrak{A}$
thus constructed is a DG Hopf algebra;
in \S\ref{subsect:quantization}
we prove the above theorem and in \S\ref{subsect:summary} 
we extend this quantization to the corresponding Lie algebra homologies.

\subsection{Construction of the Hopf
algebra}\label{subsect:construction}

Let $\mu$ be an indeterminant with degree $0$ and $\DFA H$ the tensor product $\DFA \otimes k[\mu,\mu^{-1}].$
For an element $a\otimes \mu^{u} \in \DFA \otimes k[\mu,\mu^{-1}],$ we denote it by $(a,u),$ where $u$ will be named as the height.
Put
\begin{align}\SLH[h,\hbar] \coloneqq \Sym^{\bullet}(\CC_{\bullet}^{\lambda}(\DFA H)[1])\otimes k[h,\hbar],\label{f32}\end{align}
where $\Sym^{\bullet}$ means the symmetric power, $h$ and $\hbar$ both have degree $n-2$.
Consider the $k[h,\hbar]$-submodule $\widetilde{\SLH}$ spanned by elements \begin{align}[(a_{11},h_{11}),\ldots,(a_{1p_1},h_{1p_1})]\bullet \cdots \bullet [(a_{k1},h_{k1}),\ldots,(a_{kp_k},h_{kp_k})] \text{ for }k,p_1,\ldots,p_k\in \mathbb{Z}_{\geq 0}\label{f31}\end{align} such that $h_{ij} \neq h_{i'j'}$ if $i\neq i'$ or $j\neq j'$, where $\bullet$ denotes the product in the symmetric algebra $\widetilde{\SLH}$.
We may replace each $h_{ij}$ in (\ref{f31}) with an integer $h_{ij}'$.
The new element is said to be equivalent to the original one if $h_{ij}<h_{i'j'} \Leftrightarrow h_{ij}'<h_{i'j'}'$ for all $i,j,i',j'.$
Let $\widetilde{\quan}$ be the quotient of $\widetilde{\SLH}$ by this equivalent relation.

Let $\widetilde{\quanq}$ be the $k[h,\hbar]$-ideal
of $\widetilde{\quan}$ generated by \begin{align}X-X_{i,j,i',j'}'-X_{i,j,i',j'}''\label{f38}\end{align} satisfying the
following conditions: 
\begin{itemize}[noitemsep,leftmargin=*,nosep]
    \item[$-$] $X$ has the form of the element in (\ref{f31}). We may assume that it equals the element in (\ref{f31}); 
    \item[$-$] The inequality $h_{ij}<h_{i'j'}$ holds and there exist no $i'',j''$ s.t. $h_{ij}<h_{i''j''}<h_{i'j'}$.
    The element $X_{i,j,i',j'}'$ is obtained from $X$ by only interchanging $h_{ij}$ with $h_{i'j'}$.
    \item[$-$]
    For the case $i \neq i',$ the element $X_{i,j,i',j'}''$ is obtained from $X$ by only replacing two components $[(a_{i1},h_{i1}),\ldots,(a_{ip_i},h_{i})]$ and $[(a_{i'1},h_{i'1}),\ldots,(a_{i'p_i'},h_{i'p_i'})]$ with the component
    \begin{align}\begin{alignedat}{2}
    \pm \langle a_{ij},a_{i'j'}\rangle \cdot h \cdot [&(a_{i(j+1)},h_{i(j+1)}),\ldots,(a_{i(j-1)},h_{i(j-1)}),\\
    &(a_{i'(j'+1)},h_{i'(j'+1)}),\ldots,(a_{i'(j'-1)},h_{i'(j'-1)})].\end{alignedat}\label{f36}\end{align}
    \item[$-$] For the case $i = i'$,
    the element $X_{i,j,i',j'}''$ is obtained from $X$ by only replacing the component $[(a_{i1},h_{i1}),\ldots,(a_{ip_i},h_{i})]$ with the component
    \begin{align}\begin{alignedat}{2}
    \pm \langle a_{ij},a_{ij'}\rangle \cdot \hbar \;\cdot\; &[(a_{i(j+1)},h_{i(j+1)}),\ldots,(a_{i(j'-1)},h_{i(j'-1)})]\bullet\\
    &[(a_{i(j'+1)},h_{i(j'+1)}),\ldots,(a_{i(j-1)},h_{i(j-1)})].    
    \end{alignedat}\label{f37}\end{align}
\end{itemize}
Note that the 
degrees of $X,$ $X_{i,j,i',j'}'$, and $X_{i,j,i',j'}''$ are the same.
Put
\begin{equation}\label{defofA}
\quan \coloneqq \widetilde{\quan}/\widetilde{\quanq}.
\end{equation}
Our goal is to define 
a DG Hopf algebra structure on $\quan$, which
quantizes the DG Lie biaglebra 
$\CC_{\bullet}^{\lambda}(\DFA)[1]$.

\subsubsection{The
differential algebra 
structure}

There is a differential and an algebra structure on $\quan$:
\begin{itemize}[leftmargin=*,nosep,noitemsep]
\item [$-$]
For an element $X$ of the form in 
(\ref{f31}), define the
differential $b$ on $\widetilde{\quan}$ by putting (\cf \cite{CEG}) 
\begin{align}\begin{split}
b(X)\coloneqq \sum_{\text{all }i,j}\sum_{(a_{ij})}\pm &[(a_{11},\widehat{h}_{11}),\ldots,(a_{1p_1},\widehat{h}_{1p_1})]\bullet \cdots \\ 
\bullet & [(a_{i1},\widehat{h}_{i1}),\ldots,(a_{ij}',h_{ij}),(a_{ij}'',h_{ij}+1),\ldots,(a_{ip_i},\widehat{h}_{ip_i})]\bullet \cdots,\end{split}\label{f63}\end{align}
where 
$\widehat{h}_{i'j'} = h_{i'j'}$ if $h_{i'j'} \leq h_{ij},$ and $= h_{i'j'}+1$ otherwise.
The equality $b^{2} = 0$ follows from that of $\CC_{\bullet}(A^{\ac})[1]$.

\item[$-$]
Let $X$ and $X'$ be two elements as in (\ref{f31}).
For components $h_{ij}$ of $X$ and $h_{i'j'}'$ of $X'$, put \begin{align}
C \coloneqq 1+\max_{i,j,i',j'}\{h_{ij} - h_{i'j'}'\}.\label{f310}\end{align}
Let $X''$ be the element obtained from $X'$ by replacing $h_{i'j'}'$ with $h_{i'j'}'+C$ for all $i',j'$.
Define the product of $X$ and $X'$ to be $X\bullet X''$.
The natural embedding $k \hookrightarrow \SLH[h,\hbar]$ of zeroth symmetric power induces an embedding $k[h,\hbar]\to \quan,$ which is set to be the unit.
\end{itemize}

\subsubsection{Construction
of the coproduct}\label{sec:constcoprod}

Next, we define a coalgebra structure on $\quan$.
For an element $X$ of the form in (\ref{f31}) and an integer $m\geq 2$, define the $m$-fold coproduct $\Delta_{m}(X)$.
Put $P \coloneqq \{(i,j)\mid \{1\leq i\leq k,1\leq j\leq p_i\}$.
Set $(i,j) + (0,1) = (i,j+1)$ and $(i,j) = (i,j - p_{i})$ for $j \geq p_{i} + 1$.
To define the coproduct, we first introduce the following:

\begin{definition}\label{d32}
Let $X$ be the element in (\ref{f31}).
An \emph{$m$-coloring of $X$} is a triple $(I,\phi,c)$,
where $I$ is a subset of $P,$ $\phi$ is a map $I\to I$, and $c$ is a map 
\[P \longrightarrow \{1,\ldots,m\};\quad (i,j)\longmapsto c(i,j)\] satisfying the following
conditions:
\begin{itemize}[leftmargin=*,nosep,noitemsep]
    \item [$-$] The cardinal of $I$ is even. The map $\phi$ has no fixed point and satisfies $\phi^2 = \id.$
    
    \item [$-$] The map $c$ is not necessarily surjective but satisfies:
    \begin{enumerate}[label = \rm{(}\arabic*),noitemsep,leftmargin=*,nosep]
        \item $c(i,j) = c(i,j+1)$ for $(i,j)\in P\setminus I$;
        \item for any $(i,j)\in I$,
        \begin{enumerate}[ noitemsep,leftmargin=*,nosep]
        \item $c(i,j) = c(\phi(i,j)+(0,1))$ and $c(\phi(i,j)) = c((i,j)+(0,1));$
        \item $c(i,j)\neq c(\phi(i,j))$ (hence $c(i,j) \neq c(i,j+1)$);
        \item $c(i,j) > c(\phi(i,j))$ if and only if $h_{ij} > h_{\phi(i,j)}$. Here and thereafter, set $h_{(i',j')} = h_{i',j'} = h_{i'j'}$ for $(i',j') = \phi(i,j)$.     
        \end{enumerate}
    \end{enumerate}
\end{itemize}
\end{definition}

For an $m$-coloring $(I,\phi,c)$ of $X$, define \begin{align}f: P\longrightarrow P;\quad (i,j)\longmapsto \begin{cases}(i,j+1) & (i,j)\notin I;\\
\phi(i,j)+(0,1) & \text{otherwise}
\end{cases}\label{f34}\end{align}
and $g:P\setminus I \to P\setminus I;$ $(i,j) \mapsto f^{r}(i,j)$, where $r$ is the minimal number such that $f^{r}(i,j) \in P\setminus I$.
Let $\mathcal{Q}$ be the set of orbits under iterations of $g$.
The map $c$ induces 
\[\overline{c}: \mathcal{Q} \longrightarrow \{1,\ldots,m\};\quad Q \longmapsto c(i,j)\text{ for one }(i,j)\in Q.\]
For an orbit $Q = \{(i,j),g(i,j),\ldots,(i',j')\}$ in $\mathcal{Q}$ with $g(i',j') = (i,j)$, put $X_{Q} \coloneqq [(a_{ij},h_{ij}),\allowbreak (a_{g(i,j)},h_{g(i,j)}), \allowbreak\ldots \allowbreak,\allowbreak
(a_{i'j'},\allowbreak h_{i'j'})].$
For each $1\leq i\leq m$, put
\[X_{(I,\phi,c)}^{(i)} \coloneqq \begin{cases} 
X_{Q_{i_1}} \bullet X_{Q_{i_2}} \bullet \cdots & \overline{c}^{-1}(i) = \{Q_{i_1},Q_{i_2},\ldots\}\\
1& \text{otherwise.}
\end{cases}\]
For each $m$-coloring, put
\[\epsilon_{(I,\phi,c)} = \prod \langle a_{ij},a_{\phi(i,j)}\rangle,\]
where the product extends over pairs $(i,j)\in I$ such that $c(i,j)<c(\phi(i,j)).$
Let $\mathcal{P}$ be the set of orbits under iterations of $f$ and $l$ the cardinal of $\mathcal{P}$. 
For $k$ the number of components in $X$, put $N \coloneqq k-l.$
Put 
\begin{align*}h^{(I,\phi,c)} \coloneqq h^{(\# I + 2N)/4}
\text{ and }\hbar^{(I,\phi,c)} \coloneqq \hbar^{(\#I - 2N)/4}.\end{align*}
In Lemma~\ref{integerpower},
we shall show that the 
powers of $h$ and $\hbar$ are integers.
Finally, the $m$-fold coproduct of $X$ is defined by
\begin{align}
\Delta_{m}(X) \coloneqq \sum_{(I,\phi,c)} X_{(I,\phi,c)} \coloneqq \sum_{(I,\phi,c)}\epsilon_{(I,\phi,c)} \cdot h^{(I,\phi,c)} \cdot \hbar^{(I,\phi,c)}\cdot X_{(I,\phi,c)}^{(1)}\otimes \cdots \otimes X_{(I,\phi,c)}^{(m)},\label{f33}
\end{align}
where the sum extends over all $m$-colorings of $X$. 
We give 
in Example~\ref{e31} some
more detials of
this $m$-fold coproduct.

\begin{lemma}[{c.f. \cite[Section~9]{Turaev}}]\label{integerpower}
Let $(I,\phi,c)$ be an $m$-coloring of $X$ and $N$ the integer above.
Then we have:
\begin{enumerate}[label = \rm{(}\arabic*),noitemsep,leftmargin=*,nosep]
    \item
$2$ divides $(\# I)/2 + N$ and $(\# I)/2 - N.$ 
    \item
The degree of $\epsilon_{(I,\phi,c)} \cdot h^{(I,\phi,c)} \cdot \hbar^{(I,\phi,c)}\cdot X_{(I,\phi,c)}^{(1)}\otimes \cdots \otimes X_{(I,\phi,c)}^{(m)}$ is equal to that of $X$.
\end{enumerate}
\end{lemma}

\begin{proof}
Admit (1) for a moment. 
The degree of $h^{(I,\phi,c)} \cdot \hbar^{(I,\phi,c)}$ equals $(\# I)/2\cdot(n-2)$. Notice that $(\# I)/2$ equals the number of symmetric pairings in $\epsilon_{(I,\phi,c)}$.
As each symmetric pairing results in a summand
$2-n$ of the degree (\cf the degree of the Lie bialgebra in Proposition~\ref{prop:Liebialg}), (2) follows.

We show (1) by
induction on the pairs in $I$. 
Put $(i',j') \coloneqq \phi(i,j)\in I$ throughout this proof.
Let $N^+$ (resp. $N^-$) be the integer so that $2N^+$ (resp. $2N^-$) equals the number of pairs $(i,j)\in I$ satisfying $i' = i$ (resp. $i' \neq i$).
Hence $(\# I)/2 = N^+ + N^-$ holds.

We are going to show that \begin{align}N = N^{-} - N^{+} \mod 2,\label{f61}\end{align} which suffices to imply (1).
Fix $(i,j) \in I$ and put $I^{(0)} \coloneqq \{(i,j),(i',j')\}.$
For $(I_{0},\phi|_{I_0})$ and $f_0:P\to P$ defined using $\phi|_{I_0}$, we consider the components corresponding to the orbits under the iteration of $f_0$:
\begin{enumerate}[label = \rm{1.}\arabic*),noitemsep,leftmargin=*,nosep]
    \item ($i' = i$) Two components $[\widehat{(a_{ij},h_{ij})},(a_{i(j+1)},h_{i(j+1)}),\ldots]\bullet[\widehat{(a_{ij'},h_{ij'})},\ldots]$ are obtained from the component $[(a_{i1},h_{i1}),\ldots]$ in $X$, where $\widehat{(a_{ij},h_{ij})}$ indicates that $(a_{ij},h_{ij})$ is deleted from the component.
    \item ($i'\neq i$) One component $[\widehat{(a_{ij},h_{ij})},\ldots,\widehat{(a_{i'j'},h_{i'j'})},\ldots]$ is obtained from two components $[(a_{i1},h_{i1}),\ldots]\bullet [(a_{i'1},h_{i'1}),\ldots]$ in $X$.    
\end{enumerate}
Let $X_{0}$ denote the element obtained from 1.1) or 1.2).
For the case 1.1), an $m$-coloring $(I^{(0)},\phi^{(0)},c^{(0)})$ exists  
and hence $+1$ of $N^{+}$ contributes to $+1$ for the number of components and hence leads to $-1$ of $N$.  
This is compatible with the desired equality (\ref{f61}), \ie, the number $N$ for the $m$-coloring $(I^{(0)},\phi^{(0)},c^{(0)})$ of $X$ is exactly $1.$
No $m$-coloring $(I^{(0)},\phi^{(0)},c^{(0)})$ exists for the case 1.2).

Put $I_{0} \coloneqq I\setminus I^{(0)}.$
If $I_0$ is nonempty, then fix arbitrary $(i_1,j_1)\in I_{0}$ and $(i_1',j_1') \coloneqq \phi(i_1,j_1).$
\begin{itemize}[noitemsep,leftmargin=*,nosep]
    \item [$-$]
Given 1.1) and 
$i_1 = i_1' = i = i'$, it may happen 
that $(i_1,j_1)$ and $(i_1',j_1')$ are in different components of $X_0$ (\ie Example~\ref{e31}~(1)).
Accordingly, we put $N_{1}^{+} \coloneqq N^{+} - x$ and $N_{1}^{-} \coloneqq N^{-} + x$, 
where $2x$ equals the number of all pairs $(\ihat,\jhat)$ and $(\ihat',\jhat') \coloneqq \phi(\ihat,\jhat)$ such that $\ihat = \ihat'$ but $(\ihat,\jhat)$ and $(\ihat',\jhat')$ are in different components in $X_{0}$.

\item [$-$]
Given 1.2) and $i_1 = i \neq i' = i_1'$, 
then $(i_1,j_1)$ and $(i_1',j_1')$ are in the same component of $X_0$ (cf. Example~\ref{e31}~(3)).
Accordingly, we put $N_{1}^{+} \coloneqq N^{+} + y$ and $N_{1}^{-} \coloneqq N^{-} - y$, 
where $2y$ equals the number of all pairs $(\ihat,\jhat)$ and $(\ihat',\jhat') \coloneqq \phi(\ihat,\jhat)$ such that $\ihat \neq \ihat'$ but $(\ihat,\jhat)$ and $(\ihat',\jhat')$ are in same component in $X_{0}$.
\end{itemize}
Note that (\ref{f61}) holds if and only if $N = N_1^- - N_1^+ \mod 2$ holds.

Put $I^{(1)}\coloneqq \{(i_1,j_1),(i_1',j_1')\}$. 
Consider $(I^{(1)},\phi|_{I^{(1)}})$ for $X_{0}$ and $I_{0}$ as above and let $I^{(1)}$ play a role similar to $I^{(0)}$.
Similar to cases 1.1) and 1.2), we have cases 2.1) and 2.2) for the obtained element $X_1$:
\begin{enumerate}[label = \rm{2.}\arabic*),noitemsep,leftmargin=*,nosep]
    \item 
($(i_1,j_1),(i_1',j_1')$ in the same component of $X_{0}$) 
If 1.1) happens, one may extend the $m$-coloring $c^{(0)}$ of $X_{0}$ to an $m$-coloring of $X_{1}$.
We notice that even if 1.2) happens, there might be an $m$-coloring of $X_1$ under certain conditions (\eg Example~\ref{e31}~(2)).
When such an $m$-coloring exists, $-1$ of $N^-$ contributes $+1$ of $N'$ (in 1.2)) and then $+1$ of $N_1^+$ contributes $-1$ of $N'$, which are compatible with the desired equality.

\item 
($(i_1,j_1),(i_1',j_1')$ in different components of $X_{0}$) 
No $m$-coloring of $X_1$ exists (\eg the second case of Example~\ref{e31}~(1)).
\end{enumerate}
Note that for the case where no $m$-coloring exists, one may still regard $\pm 1$ in $N_1^+$ (resp. $N_1^-$) contributes to $\mp 1$ (resp. $\pm 1$) of $N'$.
These are compatible with the desired equality.
Moreover, one may set $I^{(0)} \coloneqq \{(i_1,j_1),(i_1',j_1')\}$ and $I^{(1)} \coloneqq \{(i,j),(i',j')\}$ instead.
Then by a similar argument, we have similar case-by-case results, and moreover the same $X_{1}$ will be obtained. Here, all results are compatible with the desired equality (\ref{f61}). 

One may similarly obtain $(i_2,j_2)\in I_1 \coloneqq I\setminus I^{(1)}$, $I^{(2)} \coloneqq \{(i_2,j_2),\phi(i_2,j_2)\}$, and $N_{2}^{\pm}$. 
There are similar case-by-case results.
In general, we may apply induction with hypothesis to be that $\pm 1$ in $N_l^+$ (resp. $N_l^-$) induced by $(I^{(l)},\phi_{I^{(l)}})$ contributes to $\mp 1$ (resp. $\pm 1$) of $N'$ for $l\geq 1$.
Moreover, the desired equality holds regardless of the choice of $I^{(0)},I^{(1)},\ldots,I^{(\frac{\# I}{2})}$.
\end{proof}

We show the DG coalgebra structure on $\quan$ is well-defined (c.f. \cite[3.4]{Schedler}).

\begin{proposition}\label{mfoldcoprod}
Let $\Delta_m$ be given by 
\eqref{f33}, and let $\Delta:=
\Delta_2:\quan\to \quan\otimes \quan$.
Then $(\quan,\Delta)$
is a DG coalgebra such that
$\Delta_m=(id\otimes
\cdots\otimes \Delta_2)\circ\cdots
\circ\Delta:
\quan\to \quan^{\otimes n}$.
\end{proposition}

\begin{proof}
For $X$ in (\ref{f31}), and certain pairs $(i,j)$ and $(i',j')$ defining $\widetilde{\quanq}$, put $G \coloneqq X - X' - X'' \coloneqq X - X_{i,j,i',j'}' - X_{i,j,i',j'}''$.
We are to show $\Delta_{m}(G) = \Delta_{m}(X) - \Delta_{m}(X') - \Delta_{m}(X'') = 0$ in $\quan$.
At a high level, the proof is to find the
correspondence between $m$-colorings of $X,$ $X'$, $X''$.
Regarding each $m$-coloring $(I,\phi,c)$ of $X$, we have situations below 
(similarly for $m$-colorings of $X'$): 
\begin{enumerate}[label = \arabic*),noitemsep,leftmargin=*,nosep]
    \item Exactly one of $(i,j),(i',j')$ belongs to $I$: Now, we have $X_{(I,\phi,c)} = X_{(I,\phi,c)}'$ in $\widetilde{\quan}$ because there exists no pair $(i'',j'')$ such that $h_{i''j''}$ lies strictly between $h_{ij}$ and $h_{i',j'}.$ 
    \item Both $(i,j),(i',j') \notin I$: Now, the $3$-tuple $(I,\phi,c)$ is also an $m$-coloring of $X'$.
    \begin{enumerate}[label = 2.\arabic*),noitemsep,leftmargin=*,nosep] 
    \item Case $c(i,j) \neq c(i',j')$: Now $a_{ij}$ and $a_{i'j'}$ belong to different tensor-components in $X_{(I,\phi,c)}$. We have
    \begin{align*} X_{(I,\phi,c)} & =  \cdots \otimes [(a_{ij},h_{ij}),\ldots]\bullet \cdots \otimes \cdots \otimes [(a_{i'j'},h_{i'j'}),\ldots]\bullet \cdots \otimes \cdots\text{ and}\\
    X_{(I,\phi,c)}' & = \cdots \otimes [(a_{ij},h_{i'j'}),\ldots]\bullet \cdots \otimes \cdots \otimes [(a_{i'j'},h_{ij}),\ldots]\bullet \cdots \otimes \cdots\end{align*}
    As $[(a_{ij},h_{ij}),\ldots] = [(a_{ij},h_{i'j'}),\ldots]$ and $[(a_{ij},h_{ij}),\ldots] = [(a_{i'j'},h_{ij}),\ldots]$ in $\widetilde{\quan}$, we have $X_{(I,\phi,c)} = X_{(I,\phi,c)}'$.
    \item Otherwise: We have $(I,\phi,c|_{P''})$ is an $m$-coloring of $X''$. Here $P''\coloneqq P\setminus \{(i,j),(i',j')\}$.
    One may straightforwardly check that $X_{(I,\phi,c)} - X_{(I,\phi,c)}' - X_{(I,\phi,c|_{P''})}''$ belongs to $\widetilde{\quanq}$.
    \end{enumerate}
    \item Both $(i,j),(i',j') \in I$: \begin{enumerate}[label = 3.\arabic*),noitemsep,leftmargin=*,nosep] \item If $\phi(i,j) \neq (i',j')$, then $X_{(I,\phi,c)} = X_{(I,\phi,c)}'$ holds.
    \item Otherwise, $(I,\phi,c)$ is not an $m$-coloring of $X'$. 
    As for $X''$, with $I''=I\setminus \{(i,j),(i',j')\}$, $\phi'' = \phi|_{I''}$, and $c'' = c|_{P''},$ the $3$-tuple $(I'', \phi'',c'')$ is an $m$-coloring of $X''$.
    We may check that $X_{(I,\phi,c)} - X_{(I,\phi,c)}' - X_{(I,\phi,c|_{P''})}''$ belongs to $\widetilde{\quanq}$.
    \end{enumerate}
\end{enumerate}

Conversely, consider each $m$-coloring $(I'',\phi'',c'')$ of $X''$ and equations
\begin{align}(i',j') = (i,j+1)\text{ and }(i,j) = (i',j'+1).\label{f62}\end{align}
First, assume none of the above equations hold. 
Set $c$ to be $c(i,j) \coloneqq c''(i',j'+1)$, $c(i',j') \coloneqq c''(i,j+1)$, and $c|_{P''} \coloneqq c''$.
Next, without loss of
generality,
consider the case where $(i',j') = (i,j+1)$ in (\ref{f62}) holds.
Now 
\[X'' = \pm \langle a_{ij},a_{i'j'}\rangle \cdot \hbar \cdot \big(\cdots \bullet[\widehat{(a_{ij},h_{ij})},\cdots]\bullet [\widehat{(a_{i'j'},h_{i'j'})}]\bullet\cdots \big).\]
Set $c''$ to be $c(i,j) \coloneqq c''(i,j+2)$, $c(i',j') \coloneqq c''(\widehat{i',j'})$, and $c|_{P''}\coloneqq c''$.
If $c(i,j) = c(i',j')$, we recover the case 2.2) above with $I \coloneqq I''$ and $\phi \coloneqq \phi''.$
Otherwise, with $I \coloneqq I'' \cup \{(i,j),(i',j')\}$, $\phi(i,j) \coloneqq (i',j')$, and $\phi|_{I''} \coloneqq \phi''$, we recover the case 3.2). 

For each $m$-coloring of $X$, we have associated to it at least one $m$-coloring of $X'$ or $X''$.
For each $m$-coloring of $X''$, there is at most one $m$-coloring of $X$ and that of $X'$ associated.
Moreover, in each case, we have the equation $X_{(I,\phi,c)} = X_{(I,\phi,c)}'$ or $X_{(I,\phi,c)} = X_{(I,\phi,c)}' + X_{(I,\phi,c)''}'',$ where $(I,\phi,c)'' = (I,\phi,c|_{P''})$ in the case 2.2) or $(I,\phi,c)'' = (I'',\phi'',c'')$ in the case 3.2).
\end{proof}

\subsection{Proof of the DG Hopf algebra 
structure}
\label{subsect:Hopf}

The goal of this subsection
is to show the following.

\begin{proposition}
Let $(\quan,\bullet,\Delta)$ be as above.
Then $\quan$ is a DG Hopf algebra.
\end{proposition}

\noindent The proof consists of several steps.

(1)(Associativity and coassociativity) 
For the associativity, we need to show $(X\bullet Y) \bullet Z = X \bullet (Y\bullet Z)$ for $X,Y,Z$ of the form (\ref{f31}). 
    Regardless of all heights, this equality straightforwardly comes from the associativity of symmetric product.
    By (\ref{f310}), the order of the heights in $(X\bullet Y) \bullet Z$ is the same as that in $X \bullet (Y\bullet Z).$
    The associativity hence follows.

    As for the coassociativity, similarly to \cite[3.7]{Schedler} by considering a graph with only one vertex, we can show that 
    \begin{align*}
        (\Delta_{2}\otimes 1)\circ \Delta_2(X) = \Delta_3(X) = (1\otimes \Delta_2)\circ \Delta_2(X) \text{ for $X$ of the form (\ref{f31}).}
    \end{align*}
    The coassociativity follows immediately.

(2) (Bialgebra identity) We need to show $\Delta_2 (X\bullet Y) = \Delta_2 (X) \bullet \Delta_2 (Y)$ for $X,Y$ of the form (\ref{f31}).
The proof is similar to \cite[3.8]{Schedler}.
The key point is that there is no $2$-coloring $(I,\phi,c)$ of $X\bullet Y$ such that $\phi(i,j) = (i',j')$ for some $(i,j)\in P_X$ and $(i',j')\in P_Y$, where $P_X$ and $P_Y$ denote respectively the set of pairs $(i,j)$ in $X$ and $Y$. 
    
(3) (Connectivity) As elements in $\CC_\bullet^{\lambda} (A^{\ac}H)[1]$ have negative degree, the degree zero part of $\SLH[h,\hbar] = \Lambda^{\bullet} (\CC_\bullet^{\lambda} (A^{\ac}H)[1])\otimes k[h,\hbar]$ and hence that of $\quan$ are equal to $k$.
The bialgebra $\quan$ is connected and hence a Hopf algebra.

(4) (Compatibility with the differential) The Hopf algebra is moreover a DG Hopf algebra with the differential (\ref{f63}).
Then similarly to \cite[Lemma~14]{CEG}, we can show $b(X-X'-X'')\in \widetilde{B}$ for $X-X'-X''$ in (\ref{f38}).

\subsection{$\quan$ is a quantization of
$\mathrm{CC}_\bullet
^{\lambda}
(A^{\ac})[1]$}
\label{subsect:quantization}

In this subsection, 
we show Theorem~\ref{theoremsec6}, \ie $\quan$
is a quantization of
$\mathrm{CC}_\bullet
^{\lambda}
(A^{\ac})[1]$ in the sense of Definition~\text{\rm{\ref{d31}}}.
We first check the condition (2) in Definition~\ref{d31}.

\begin{lemma}
For the Hopf algebra $\quan$, 
$x,y\in \CC_{\bullet}^{\lambda}(A^{\ac})[1]$ 
and their lifting $\tilde x,\tilde y\in \quan$, we have 
\[\frac{\tilde x \tilde y - \tilde y \tilde x}{h} = \{x,y\} \mod (h,\hbar)\text{ and }\frac{\Delta(\tilde x) - \Delta^{op}(\tilde x)}{\hbar} = \delta(x) \mod (h,\hbar).\]
\begin{proof}
Put 
\begin{align*}
\tilde x = [(a_{1},h_{1}),\ldots,(a_{p},h_{p})] \text{ and }\tilde y = [(a_{1}',h_{1}'),\ldots,(a_{q}',h_{q}')].
\end{align*}
We prove the left equation.
Because of (\ref{f37}) and modulo $\hbar$, the relation (\ref{f38}) in $\quan$ writes $X = X_{i,j,i,j'}'$.
This implies that we only need to show the equation for those $\tilde x$ and $\tilde y$ such that $h_{ij}<h_{i(j+1)}$ and $h_{ij}'<h_{i(j+1)}'$ for all possible $i,j$.
Assume without
loss of generality
that $h_{p}<h_{1}'$ and $p > q$.
In this case, we have
\begin{eqnarray}\tilde x\tilde y & = & [(a_{1},h_{1}),\ldots,(a_{p-1},h_{p-1}),(a_{p},h_{1}')] \bullet [(a_{1}',h_{p}),(a_{2}',h_{2}'),\ldots,(a_{q}',h_{q}')] \nonumber \\
& & \pm \langle a_{p},a_{1}'\rangle \cdot h \cdot [\cdots]\quad\quad \text{ (here $[\cdots]$ follows (\ref{f36}))} \nonumber \\
& = & [(a_{1},h_{1}),\ldots,(a_{p-1},h_{p-1}),(a_{p},h_{2}')] \bullet [(a_{1}',h_{p}),(a_{2}',h_{1}'),\ldots,(a_{q}',h_{q}')] \nonumber \\
& & + \sum_{i=1,2} \pm \langle a_{p},a_{i}'\rangle \cdot h \cdot [\cdots] \nonumber \\
& = & [(a_{1},h_{1}),\ldots,(a_{p-1},h_{p-1}),(a_{p},h_{q}')] \bullet [(a_{1}',h_{p}),(a_{2}',h_{1}'),\ldots,(a_{q}',h_{q-1}')] \label{quantizeqn}\\
& & + \sum_{i=1,\ldots,q} \pm \langle a_{p},a_{i}'\rangle \cdot h \cdot [\cdots] \nonumber \\
& = &[(a_{1},h_{q+1}),\ldots,(a_{p-q},h_{p}),(a_{p-q+1},h_{1}'),\ldots,(a_{p},h_{q}')] \nonumber \\
& & \bullet [(a_{1}',h_{1}),\ldots,(a_{q}',h_{q})] + h \cdot \sum_{\substack{j=1,\ldots,p\\i=1,\ldots,q}}\pm \langle a_{(p+1)-j},a_{i}'\rangle \cdot [\cdots]. \nonumber 
\end{eqnarray}
Here, to get the last equation above, we follow (\ref{quantizeqn}) to get $(a_{p-1},h_{q-1}'),$ $(a_{p-2},h_{q-2}'),$ $\ldots,$ $(a_{1},h_{q+1})$ step by step.    

Due to the assumption, we have 
\[[(a_{1},h_{q+1}),\ldots,(a_{p},h_{q}')]\bullet [(a_{1}',h_{1}),\ldots,(a_{q}',h_{q})] = \tilde y \tilde x\text{ in }\widetilde{\quan}\] and hence the left equation is proved.

We next show the right equation.
By (\ref{f38}) and (\ref{f37}), we have for any $1\leq i\leq j\leq n$
\[\tilde{x} = \tilde{x}' \pm \langle a_{i},a_{j} \rangle \cdot \hbar \cdot X''\]
given no $h_{j'}$ lies strictly between $h_{i}$ and $h_{j}$, where $\tilde{x}'$ and $X''$ follows (\ref{f37}).
We have 
\[\Delta(\hbar \cdot X'') = 1 \otimes \hbar \cdot X'' + \hbar \cdot X'' \otimes 1 \mod (h,h\hbar,\hbar^2).\]
Hence $\Delta(\hbar \cdot X'') - \Delta^{op}(\hbar \cdot X'') = 0 \mod (h,h\hbar,\hbar^2).$

By the previous paragraph, when considering $\Delta(\tilde x) - \Delta^{op}\tilde x) \mod (h,h\hbar,\hbar^2)$, we may assume that $h_{i} < h_{j}$ if $i < j.$
For such $\tilde x,$ we have
\begin{align*}
\Delta(\tilde x) & = 1\otimes \tilde x + \tilde x\otimes 1 + \hbar\cdot\sum_{i<j} \pm \langle a_{i},a_{j}\rangle [\cdots]\otimes[\cdots] \mod (h,h\hbar,\hbar^2),\end{align*}
where $[\cdots]\otimes[\cdots] = [(a_{i+1},h_{i+1}),\ldots,(a_{j-1},h_{j-1})]\otimes [(a_{j+1},h_{j+1}),\ldots,(a_{i-1},h_{i-1})].$
The desired equality thus follows.
\end{proof}
\end{lemma}

Let $L$ denote $\CC_{\bullet}^{\lambda}(A^{\ac})[1]$ and $\mathrm{SL}$ denote the symmetric algebra $\Lambda^{\bullet}(L)$.
By Definition~\ref{d31}, we need to check the isomorphism 
$\quan/(h\cdot \quan +\hbar \cdot \quan) \cong \mathrm{SL}.$
This can be implied by the isomorphism $\quan \cong \mathrm{SL}[h,\hbar]$, which will be proved as a corollary of Proposition~\ref{p31}. 
This proposition is an analogue to the Poincar\'e-Birkhoff-Witt theorem.

Let $B_{A^{\ac}}$ denote a basis of $A^{\ac}.$
Consider the set 
\begin{align*}
B_{\quan} & \coloneqq \left\{
X_{1}\bullet X_{2} \bullet \cdots \bullet X_{k} \,\left|\,X_{i} =[(a_{i1},1),\cdots,(a_{ip_{i}},p_{i})],
a_i\in B_{A^{\ac}}
\right.\right\}\subset \SLH[h,\hbar].
\end{align*}

\begin{proposition}\label{p31}
The set $B_{\quan}$ is a basis of $\quan$ over $k[h,\hbar]$.
\end{proposition}

For its proof, we use a result \cite[Lemma~4.4]{Schedler} of the Diamond lemma.
Recall Schedler's definition of lifting.
\begin{enumerate}[label = \rm{(L.}\arabic*),noitemsep,leftmargin=*,nosep]
    \item 
Let $\mathrm{Map}(\mathbb{N}\times \mathbb{N},\mathbb{N})$ denote the set of all maps $\mathbb{N}\times \mathbb{N} \to \mathbb{N}$ and $P(\mathbb{N}\times \mathbb{N})$ the power set of $\mathbb{N}\times \mathbb{N}$.
Let $\widetilde{\quan}'$ be the $k[h,\hbar]$-submodule of \[\widetilde{\quan}\times P(\mathbb{N}\times \mathbb{N}) \times \mathrm{Map}(\mathbb{N}\times \mathbb{N},\mathbb{N})\]
spanned by $(X,P_{X},\psi_{X})$ that satisfy: (1) $X$ has the form (\ref{f31}); (2) $P_{X} = \{(i,j)\mid 1\leq i\leq k,1\leq j\leq p_i\}$; (3) $\psi_{X}:P_{X} \to \{1,\ldots,|X|\};$ $(i,j) \mapsto h_{ij},$ where $|X|$ denotes the total number of entries $X$.
Different pairs $(X,P_X,\phi_X)$ are defined for different representatives of $X$, \eg if $Y = X_{1}'\bullet X_2 \cdots \bullet X_{k}$ with $X_{1}'= [(a_{1j},h_{1j}),\ldots,(a_{1(j-1)},h_{1(j-1)})]$, we have $Y = X$ in $\widetilde{\quan}$, $P_{Y} = P_{X}$ but $\psi_{Y}(1,1) = h_{1j}.$
    \item \label{l2}
Let $\quan' = \widetilde{\quan}'/\widetilde{\quanq}'.$
Here, to define $\widetilde{\quanq}'$ as a $k[h,\hbar]$-submodule of $\widetilde{\quan}'$, consider $X' = X_{i,j,i',j'}'$ and $X'' = X_{i,j,i',j'}''$ for all $X$ as in (\ref{f36}) or (\ref{f37}).
For each $X$, let $P_{X'} = P_{X}$ and $P_{X''} = P_{X}\setminus \{(i,j),(i',j')\}.$ 
Let $\psi_{X'} \coloneqq \psi_{X}$ except for putting $\psi_{X'}(i,j) = \psi_{X}(i',j')$ and $\psi_{X'}(i',j') = \psi_{X}(i,j).$ 
Finally, by taking the restriction of $\psi_{X}$ on $P_{X''}$ and then and shift down the heights without
changing the order, we obtain $\psi_{X''}$.
Let $\widetilde{\quanq}'$ be the submodule of $\widetilde{\quan}'$ generated by elements $(X,P_X,\psi_X) - (X',P_X',\psi_{X'}) -(X'',P_X'',\psi_{X''})$ with $h_{i'j'} = h_{ij} \pm 1$ in $X$. 
\end{enumerate}

Using the Diamond lemma, Schedler proved
\begin{lemma}[\text{\cite[Lemma~4.4]{Schedler}}]\label{l31}
Consider the
triples 
$(X,P_{X},\psi_{X})$ such that $X\in B_{\quan},$ $P_{X}$ has the lexicographic order, and $\psi_{X}$ satisfies $\psi_{X}(i,j) < \psi_{X}(i',j')$ if and only if $(i,j)<(i',j').$ 
Then the triples 
$(X,P_{X},\psi_{X})$ for all $X\in B_{\quan}$ form a basis of the free $k[h,\hbar]$-module $\quan'.$ 
\end{lemma}

Note that each element in $B_{\quan}$ is an equivalence class under the cyclic action. 
In the lemma, we regard that $X\in B_{\quan}$ varies from all possible representatives in equivalence classes.
An immediate result of Lemma~\ref{l31} is that elements in $B_{\quan}$ generate $\quan$.

To show Proposition~\ref{p31}, it suffices to show that elements in $B_{\quan}$ are $k[h,\hbar]$-linearly independent. 
Our goal is to obtain a $k[h,\hbar]$-linear map $\quan\to \mathrm{SL}[h,\hbar]$ so that the images of elements in $B_{\quan}$ are linearly independent.
For this, we show
\begin{lemma}[\text{cf. \cite[Lemma~4.4]{Schedler}, \cite[Lemma~9.5]{Car}}]\label{l32}
There is a $k[h,\hbar]$-linear map $\theta:\widetilde{\quan}'\to \mathrm{SL}[h,\hbar]$ such that for all generators $(X,P_{X},\psi_X)$ of $\widetilde{\quan}'$, we have:
\begin{enumerate}[label = \rm{(}\arabic*),noitemsep,leftmargin=*,nosep]
    \item $\theta(X,P_X,\psi_X) = [a_{11},\ldots,a_{1p_{1}}]\bullet \cdots \bullet [a_{k1},\ldots,a_{kp_{k}}]$ for $(X,P_X,\psi_X)$ in Lemma~\text{\rm{\ref{l31}}}; 
    \item $\theta(X,P_X,\psi_X) = \theta((X',P_X',\psi_{X'})+(X'',P_X'',\psi_{X''}))$ for triples as in \ref{l2}. 
    \item The map $\theta$ factors through $\widetilde{\quan},$ \ie the image of $(X,P_X,\psi_X)$ only depends on $X$.
\end{enumerate}
\end{lemma}

\begin{proof}[Proof of Lemma~\text{\rm{\ref{l32}}}]
Let $\widetilde{\quan}^{\prime\,n}$ be the $k[h,\hbar]$-submodule generated by $(X,P_X,\psi_X)$ such that $X$ is of the form (\ref{f31}) and $|X| = \sum_{i} p_i \leq n$.
Let $d_{(X,P_X,\psi_X)}$ denote the integer such that the number of quadruples $(i,j,i',j')$ such that $(i,j) < (i',j')$ in lexicographic order but $\psi_{X}(i,j) > \psi_X(i',j')$.
Consider the $k[h,\hbar]$-submodule $\widetilde{\quan}^{\prime\,n,d}$ generated by $(X,P_X,\psi_X)$ with ($|X|\leq n-1$) or ($|X|\leq n$ and the associated integer $d_{(X,P_X,\psi_X)} \leq d$). For instance, $(X,P_X,\psi_X)\in \widetilde{\quan}^{\prime\,n,0}$ with $|X| = n$ implies that $\psi_X(i,j) < \psi_X(i',j')$ if and only if $(i,j) < (i',j')$ in lexicographic order.

We define $\theta$ and prove that it factors through $\widetilde{\quan}$ by induction on $n$ and $d$.
On $\widetilde{\quan}^{\prime\,0,0}$, this map is trivial and factors through $\widetilde{\quan}.$
For a positive integer $n$, assume $\theta: \widetilde{\quan}^{\prime\,i} \to \mathrm{SL}[h,\hbar]$ has been defined and factors through $\widetilde{\quan}$.
Put 
\[\theta:\widetilde{\quan}^{\prime\,n,0} \to \mathrm{SL}[h,\hbar];\quad (X,P_X,\psi_X) \mapsto [a_{11},\ldots,a_{1p_1}]\bullet \cdots\bullet [a_{k1},\ldots,a_{kp_k}].\]
Here, the image is obtained from $X$ by omitting heights 
and hence only depends on $X$.
Moreover, (1) holds.

Next, assume $\widetilde{\quan}^{\prime\,n,d} \to \mathrm{SL}[h,\hbar]$ has already defined and factors through $\widetilde{\quan}.$ 
We define $\widetilde{\quan}^{\prime\,n,d+1} \to \mathrm{SL}[h,\hbar]$.
For $(X,P_X,\psi_X)$ with $d_{(X,P_X,\psi_X)} = d+1$, there exist two pairs $(i,j),(i',j')$ such that $\psi_X(i,j) = \psi_X(i',j')+1.$
Define 
\[\theta:\widetilde{\quan}^{\prime\,n,d} \to \mathrm{SL}[h,\hbar];\quad \theta(X,P_X,\psi_X) = \theta(X',P_{X'},\psi_{X'}) + \theta(X'',P_{X''},\psi_{X''}),\]
where the triples are as in \ref{l2}.
By hypothesis, $\theta(X',P_{X'},\psi_{X'})$ and $\theta(X'',P_{X''},\psi_{X''})$ only depend on $X'$ and $X''$. 
Hence if $\theta$ is well defined, then $\theta(X,P_{X},\psi_{X})$ only depends on $X$, which implies the lemma.

We need to check that the extension of $\theta$ to $\widetilde{\quan}^{\prime n,d+1}$ is well defined.
This is because there might be pairs $(i'',j'')<(i''',j''')$ satisfying $\psi_X(i'',j'') = \psi_X(i''',j''')+1$ and one can define $\theta$ using $(i'',j'')$ and $(i''',j''')$ instead. 
Note that the set $\{(i,j),(i',j'),(i'',j''),(i''',j''')\}$ may contain three or four elements.
Consider the four elements case.
We write $\theta(X')$ instead of $\theta(X',P_{X'},\psi_{X'})$ as the value $\theta(X',P_{X'},\psi_{X'})$ only depends on $X'$. 
We have
\begin{eqnarray}
\theta(X') + \theta(X'') 
 = \theta((X^{\prime})_{i'',j'',i''',j'''}^{\prime} + (X^{\prime})_{i'',j'',i''',j'''}^{\prime\prime}+(X^{\prime\prime})_{i'',j'',i''',j'''}^{\prime}+(X^{\prime\prime})_{i'',j'',i''',j'''}^{\prime\prime}),\label{f39}
\end{eqnarray}
We need to check this equals $\theta(X_{i'',j'',i''',j'''}') + \theta(X_{i'',j'',i''',j'''}'').$ 
This is straightforward (\cf \cite[Equation (4.3)]{Schedler}).

It remains to check the three elements case. 
Put $(i''',j''') = (i,j)$.
Suppose $(i,j)<(i',j')<(i'',j'')$ and we may regard $(\psi_{X}(i,j),\allowbreak\psi_{X}(i',j'),\allowbreak\psi_{X}(i'',j'')) = (3,2,1).$
There are two ways to define $\theta(X,P_{X},\psi_{X})$ and the equivalence is proved by the two equations below:   
\begin{eqnarray*}
    \theta(X,P_{X},\psi_{X}) & = & \theta(X_{i',j',i'',j''}' + X_{i',j',i'',j''}'')\\
    & = & \theta((X_{i',j',i'',j''}')_{i,j,i'',j''}') + (X_{i',j',i'',j''}')_{i,j,i'',j''}'') + X_{i',j',i'',j''}'')\\
    & = & \theta((X_{i',j',i'',j''}')_{i,j,i'',j''}' + X_{i,j,i'',j''}'' + X_{i',j',i'',j''}'')\\
    & = & \theta\left(\begin{aligned}&((X_{i',j',i'',j''}^{\prime})_{i,j,i'',j''}^{\prime})_{i,j,i',j'}^{\prime} + ((X_{i',j',i'',j''}^{\prime})_{i,j,i'',j''}^{\prime})_{i,j,i',j'}^{\prime\prime}\\
    & + X_{i,j,i'',j''}^{\prime\prime} + X_{i',j',i'',j''}^{\prime\prime}\end{aligned}\right)\\
    & = & \theta\left(\begin{aligned}&((X_{i',j',i'',j''}^{\prime})_{i,j,i'',j''}^{\prime})_{i,j,i',j'}^{\prime} + X_{i,j,i',j'}^{\prime\prime}\\
    & + X_{i,j,i'',j''}^{\prime\prime} + X_{i',j',i'',j''}^{\prime\prime}\end{aligned}\right)
\end{eqnarray*}
and similarly
\[
\theta(X,P_{X},\psi_{X})  =  \theta\left(\begin{aligned}&((X_{i,j,i',j'}^{\prime})_{i,j,i'',j''}^{\prime})_{i',j',i'',j''}^{\prime} + X_{i',j',i'',j''}^{\prime\prime}\\
    & + X_{i,j,i'',j''}^{\prime\prime} + X_{i,j,i',j'}^{\prime\prime}\end{aligned}\right).\qedhere
\]
\end{proof}

\begin{proof}[Proof of Proposition~\text{\rm{\ref{p31}}}]
Before Lemma~\ref{l32}, we have seen that $B_{\quan}$ is a generating set of $\quan$.
We need to show that elements in $B_{\quan}$ are $k[h,\hbar]$-linearly independent in $\quan.$
We have a $k[h,\hbar]$-linear map $\theta: \widetilde{\quan}\to \mathrm{SL}[h,\hbar]$ due to Lemma~\ref{l32}.
Moreover, by Lemma~\ref{l32}~(1) and (2), we know that $\theta$ induces $\overline{\theta}:\quan = \widetilde{\quan}/\widetilde{\quanq}\to \mathrm{SL}[h,\hbar]$ which is also $k[h,\hbar]$-linear.
Because the image of elements in $B_{\quan}$ under $\overline{\theta}$ is $k[h,\hbar]$-linearly independent, these elements are $k[h,\hbar]$-linearly independent in $\quan$, as desired.
\end{proof}

We next check the condition (1) in 
Definition~\ref{d31}: 

\begin{lemma}
There is an isomorphism 
$\mathrm{SL} \overset{\cong}{\to} \quan/[h,\hbar]\quan$.
\end{lemma}

\begin{proof}
For a basis $B_{A^{\ac}}$ of $A^{\ac}$, put 
\[B_{\mathrm{SL}} \coloneqq \left\{x_{1}\bullet x_{2} \bullet \cdots \bullet x_{k} \,\left|\,x_{i} =[a_{i1},\ldots,a_{ip_{i}}],a_{ij}\in B_{A^{\ac}}\right.\right\} \subset \mathrm{SL}.\]
This set is a basis of $\mathrm{SL}.$
We are find a map which maps this basis to $B_{\quan}$.
Regarding $\quan/[h,\hbar]\quan$, due to the relations defined by (\ref{f38}), the heights of the elements in $\quan/[h,\hbar]\quan$ do not matter, \ie we have $X = X_{i,j,i',j'}'$ for any $X\in \quan/[h,\hbar]\quan$ and $X_{i,j,i',j'}'$ as in (\ref{f38}).
By the universal property of symmetric algebras, the natural embedding $L\to \quan/[h,\hbar]\quan;\,\,[a_{1},\ldots,a_{p}]\mapsto [(a_1,1),\ldots,(a_{p},p)]$ induces a morphism of Hopf algebras
\[\mathrm{SL} \to \quan/[h,\hbar]\quan;\quad x_{1}\bullet \cdots \bullet x_{k} \mapsto X_{1} \bullet \cdot \cdots \bullet X_{k}\]
for $x_{i} = [a_{i1},\ldots,a_{ip_{i}}]$ and $X_{i} = [(a_{i1},1),\ldots,(a_{ip_{i}},p_{i})]$ as in $B_{\mathrm{SL}}$ and $B_{\quan}$.
As $B_{\quan}$ is a basis of $\quan$ over $k[h,\hbar]$, this map is an isomorphism. 
\end{proof}

\subsection{Summary of the above
results}
\label{subsect:summary}

Let us summarize the above results:
we constructed 
in \S\ref{subsect:construction}
a vector space
$\quan$ by \eqref{defofA}, 
then in \S\ref{subsect:Hopf}
we showed that $\quan$ is a DG Hopf
algebra, and finally
in \S\ref{subsect:quantization}
we proved that
such an $\quan$ quantizes
the Lie bialgebra structure
on $\mathrm{CC}^\lambda_\bullet(A^{\ac})$.
Combining them together, we
obtain:

\begin{theorem}
\label{thm:deformationquantization}
Suppose $A^!$ is a Frobenius algebra
and $A^{\ac}$ is its linear dual 
co-Frobenius coalgebra.
Then the Lie algebra cohomology of
$\mathfrak{gl}(A^!)$,
or equivalently,
the Lie algebra homology
of the Lie coalgebra
$\mathfrak{gl}(A^!)$,
admits a deformation quantization,
where the Lie bracket and cobracket
on the primitive elements are 
given by
\eqref{eq:bracket} and 
\eqref{eq:cobracket}
respectively. Moreover, via
the Loday-Quillen-Tsygan map,
such a Lie bialgebra is isomorphic 
to the one on
the cyclic cohomology of $A^!$,
or dually, the cyclic homology of
$A^{\ac}$.
\end{theorem}

In the light of Theorem 
\ref{cor:isoHmlgyofLiecylicdeformed},
we see that
those objects in Theorem
\ref{cor:isoHmlgyofLiecylic}
can not only be deformed,
but also be quantized.
Let us
denote by
$\mathrm{QH}(
\mathfrak{gl}^c(A^{\ac}))
)$
and
$\mathrm{QH}(
\mathfrak{gl}(A))
)$
the quantizations
of
$\mathrm{DH}(
\mathfrak{gl}^c(A^{\ac}))
)$
and
$\mathrm{DH}(
\mathfrak{gl}(A))
)$
in \eqref{iso:squareofLieinCYdeformed} 
respectively (here ``Q"
means ``quantized").
Also, denote by
$\mathfrak{A}(A^{\ac})$
and
$\mathfrak{A}(A)$
the quantizations
of
$V_{h\hbar}(\mathrm{HC}_\bullet(A^{\ac})[1])$
and
$V_{h\hbar}
(\mathrm{HC}_\bullet(A)[1])$
respectively.
Then
we get the following:

\begin{theorem}\label{cor:quantization}
Suppose $A$ is a Koszul Calabi-Yau algebra
of dimension $n$.
Denote by $A^{\ac}$ its Koszul
dual coalgebra.
Then
$$
\xymatrixcolsep{4pc}
\xymatrix{
\mathrm{QH}_\bullet(
\mathfrak{gl}^c(A^{\ac}))
\ar[r]^-{\textup{Koszul}}_{\cong}
&
\mathrm{QH}_\bullet(
\mathfrak{gl}
(A))\ar[d]^{\textup{LQT}}_{\cong}
\\
\mathfrak{A}(A^{\ac})
\ar[u]^{\textup{LQT}}_{\cong}
\ar[r]^-{\textup{Koszul}}_{\cong}&
\mathfrak{A}(A)
}
$$
is a commutative diagram of isomorphisms of
Hopf algebras,
which quantizes the commutative diagram
\eqref{iso:squareofLieinCYdeformed} in
Theorem \text{\rm{\ref{cor:isoHmlgyofLiecylicdeformed}}}.
\end{theorem}

In other words, we get a quantization
of the Loday-Quillen-Tsygan isomorphism.
Combining it with Hennion's tangent map,
such a quantization also
gives a quantization of the 
tangent map induced from the homomorphism
$\mathrm{BGL}(A)\to\mathrm{K}(A)$.
\section{Examples and applications}
\label{sect:application}

In this last section, we give a couple of examples of Calabi-Yau algebras/categories that come from algebra and geometry. In all these examples, we first show the existence of the Calabi-Yau property on these geometric objects, which is highly nontrivial in general, and then show that the results in the previous sections (especially the deformation and quantization) apply to these examples.

\subsection{Preprojective algebras}

Let $Q$ be a quiver with vertex set $Q_0$ and arrow set $Q_1$.
If $a$ is an arrow, then write $h(a)$ and $t(a)$ as the head and tail of $a$ respectively.
Let $\overline Q$ be the double of $Q$, which is obtained from $Q$ by adjoining an arrow $a^*$ with reverse orientation for each arrow $a\in Q_1$.
In what follows, we write $(a^*)^*=a$ and also $\varepsilon(a)=1$ if $a\in Q_1$ and $=-1$ otherwise. For any edge $a$, we write
$h(a)$ and $t(a)$ to be head and tail of $a$
respectively.
Fix a {\it weight} $\lambda\in k^{Q_0}$.
The corresponding {\it deformed preprojective algebra} is the quotient algebra
\begin{equation}
\Pi^\lambda(Q):=
k\overline{Q}/(\rho_\lambda),
\quad\mbox{where}\quad
\rho_\lambda=\sum_{a\in
\overline Q_1}
\varepsilon(a)aa^*
-\sum_{i\in Q_0}\lambda_i e_i.
\end{equation}
Here $k\overline Q$ is the path algebra of $\overline Q$, and the grading is given by the length of the path.

Deformed preprojective algebras were introduced by Crawley-Boevey and Holland in \cite{CBH}. In the definition, if $\lambda=0$, then $\Pi^0(Q)$ is the  {\it preprojective algebra} studied by Bernstein, Gelfand, and Ponomarev.
It is now well-known that if $Q$ is of non-Dynkin type, then the deformed preprojective algebras $\Pi^\lambda(Q)$ are Calabi-Yau of dimension 2 (see \cite{CBK}), and in particular $\Pi^0(Q)$ is Koszul Calabi-Yau.

\begin{proposition}
[{\cite{Mar,Bock}}]
If $Q$ is a non-Dynkin type quiver, then $\Pi^0(Q)$ is a Koszul algebra, and its Koszul dual algebra, denoted by $A(Q)$, is the graded algebra $k\overline Q$ modulo the following relations: 
\begin{enumerate}[label = \rm{(}\arabic*),noitemsep,leftmargin=*,nosep]
\item $aa'=0$ if $a'\ne a^*$;

\item $\varepsilon(a) a a^*=\varepsilon(b)b b^*$ if $t(a)=t(b)$. 
\end{enumerate}
Moreover, $\Pi^0(Q)$ is Calabi-Yau of dimension $2$.
\end{proposition}

We can dually give the Koszul dual coalgebra of $\Pi^0(Q)$, which we denoted by $A^{\ac}(Q)$, as follows: We first formally add to each vertex $e_i$ of $\overline Q$ an empty circle, say $o_i$, with degree $2$. 
Denote the set of $o_i$'s by $Q_2$. 
Then $A^{\ac}(Q)$ is the coalgebra spanned by $Q_0, \bar Q_1, Q_2$ with coproduct given by \begin{align*} &\Delta(e_i)=e_i\otimes e_i,\quad \Delta(a_i)=t(a_i)\otimes a_i +a_i\otimes h(a_i),\\ &\Delta(o_i)= o_i\otimes e_i+\sum_{a\in\overline Q_1:t(a)=e_i} a\otimes a^*+e_i\otimes o_i, \end{align*} where $e_i\in Q_0, a_i\in\overline Q_1,o_i\in Q_2$ respectively. 
Here, we have to notice that both $\Pi^0(Q)$ and $A^{\ac}(Q)$ are over the semi-simple ring $\oplus_{e\in Q_0}k e$. 
The main results of the paper remain valid. 

\begin{corollary}
For the preprojective algebra
$\Pi^0(Q)$, 
$$\mathrm{H}_\bullet
(\mathfrak{gl}(\Pi^0(Q)))
\cong
\mathbf{\Lambda}(\mathrm{HC}_\bullet(
A^{\ac}(Q)
)[1])\quad\mbox{and}\quad
\mathbb T_{\mathrm K(\Pi^0(Q)),0}\cong
\mathrm{CC}_\bullet(A^{\ac}(Q)).$$
The isomorphisms
can be deformed to be
isomorphisms of co-Poisson
bialgebras, where
the corresponding Lie bialgebra
comes 
from the Calabi-Yau property of
$\Pi^0(Q)$. Moreover, there
is a quantization of these 
isomorphisms.
\end{corollary}

\subsection{Fukaya categories}

Our next example will be the Fukaya category of some symplectic manifolds. There has also been some interest in the K-theory of Fukaya categories; see, for example, \cite{Mathoverflow} for some relevant discussions.

Suppose $X$ is a symplectic manifold. Intuitively, the Fukaya category of $X$ is an $A_\infty$-category whose objects are the Lagrangian submanifolds and whose morphisms are the (Lagrangian) Floer complexes.
If two Lagrangian submanifolds are transversal, the one may take the Floer complexes to be the linear span of their intersection points; if they are not transversal, then one must perturb the submanifolds by a Hamiltonian flow so that they become transversal and then take the linear span of their intersection points.
Fukaya first noticed that the ``category" thus defined is not a category, but an $A_\infty$-category; that is, a category up to homotopy. First observed by Fukaya and later developed by Seidel, Abouzaid, and many others, Fukaya categories have been a central subject in studying symplectic topology and mathematical physics.
We refer the interested reader to the recent long paper of Ganatra \cite{Gan} for an up-to-date status on this topic.

According to the types of Lagrangians submanifolds, there are two types of Fukaya categories, one is the usual Fukaya category, the other is called the wrapped Fukaya category. The following is taken from \cite[Theorems 1.12 \& 1.16]{Gan}: 

\begin{proposition}
\label{prop:Fukaya}
Let $X$ be a symplectic manifold admitting additional conditions (see \text{\rm{\cite[Assumption 3.10]{Gan}}} for more details). Then the Fukaya category and wrapped Fukaya category are well-defined for $X$: 

\begin{enumerate}[label = \rm{(}\arabic*),noitemsep,leftmargin=*,nosep]
\item 
The Fukaya category $\mathscr F(X)$ of $X$ has objects consisting of compact Lagrangian manifolds satisfying suitable hypotheses (such as being monotone), and the morphisms are the Floer cohain complexes. In this case, $\mathscr F(X)$ is a proper Calabi-Yau $A_\infty$-category. 

\item If $X$ is non-compact, then the wrapped Fukaya category $\mathscr W(X)$ of $X$ is well-defined, whose objects consist of exact Lagrangian submanifolds with some restrictions at infinity, and whose morphisms are again the Floer cochain complexes. In this case, $\mathscr W(X)$ is a smooth Calabi-Yau $A_\infty$-category.
\end{enumerate}
\end{proposition}

In general, Fukaya categories are complicated to compute.
The following example is due to Fukaya, Seidel, and Smith (\cite{FSS}) and Abouzaid (\cite{Abo1}).

\begin{example}[Fukaya category of cotangent
bundles]Suppose $N$ is 
a simply-connected compact manifold, which
is also spin.
Let $T^*N$ be the cotangent bundle of $N$,
equipped with the canonical symplectic 
structure. 
Then $\mathscr F(T^*N)$ 
is quasi-isomorphic to the singular cochain complex $C^\bullet(N)$
(see \cite{FSS}),
while the subcategory of $\mathscr W(T^*N)$ generated by a cotangent fiber, denoted by $\mathscr W(T_p^*N)$, is quasi-isomorphic to $C_\bullet(\Omega N)$, the singular chain complex of the base loop space of $N$
(see \cite{Abo1}). 

The well-known result of Adams
says that
$C_\bullet(\Omega N)$ is quasi-isomorphic to the cobar construction $\Omega(C_\bullet(N))$, and hence is Koszul dual to the singular chain complex $C_\bullet(N)$; 
therefore $\mathscr W(T_p^*N)$ is Koszul dual to $C_\bullet(N)$ in the sense of Holstein and Lazarev (see \cite{HolLaz} or 
\S\ref{subsect:gen}). 
We may also say that $\mathscr W(T_p^*N)$ and $\mathscr F(T^*N)$ are Koszul dual as Calabi-Yau $A_\infty$ categories, since $\mathscr F(T^*N)$ is quasi-isomorphic to $C^\bullet(N)$. 
\end{example}

Such a Koszul duality pattern also occurs in some other cases; see, for example,
\cite{EtLe,Li}.

\begin{corollary}
Suppose $X$ is an open symplectic manifold such that $\mathscr{F}(X)$ and $\mathscr{W}(X)$ are Koszul dual. Then the tangent complex to the K-theory of $\mathscr{W}(X)$ has a Lie bialgebra structure, which is isomorphic to the one on the cyclic cohomology of $\mathscr{F}(X)$. Moreover, such a Lie bialgebra can be quantized.
\end{corollary}

\begin{remark}[Nontriviality]
The above example of Abouzaid says that the Fukaya category of the zero section in a cotangent bundle is equivalent to the singular cochain complex of that manifold, and hence their cyclic cohomology groups are isomorphic. 
Note that the cyclic cohomology of the singular cochain complex is isomorphic to the $S^1$-equivariant homology of the free loop space of the manifold, provided it is simply-connected. 
In string topology (see \cite{ChSu,ChSu2}), Chas and Sullivan first constructed the Lie bialgebra structure on the $S^1$-equivariant homology of the free loop spaces, which is a higher dimensional generalization of the Lie bialgebra of Goldman (\cite{Gol}) and Turaev (\cite{Turaev}) on the loops on a Riemann surface. 
According to Turaev, the quantization of the Lie bialgebra on the loops is isomorphic to the skein algebra of knots in the three-dimensional space, which is the surface times the unit interval. Thus, as a corollary, the Lie bialgebra and its quantization of the Fukaya category may be viewed as a categorical generalization of Turaev and Chas-Sullivan.
\end{remark}

\appendix
\section{An example of the
iterated coproduct}

The $m$-fold coproduct
$\Delta_m$
given in \S\ref{sec:constcoprod}
is quite complicated. In 
this Appendix, we explain
its construction
via a
concrete example.

\begin{example}\label{e31}
Suppose $I = \{(i,j),(i',j'), (\ihat,\jhat), (\ihat',\jhat')\}$ with $i = \ihat$, $(i',j') = \phi(i,j)$, and $(\ihat',\jhat')=\phi(\ihat,\jhat)$. 
There are three possibilities for $I$:

\begin{enumerate}[label = \rm{(}\arabic*),noitemsep,leftmargin=*,nosep]
    \item
(the case $i' = i$ and $\ihat' = i$) First suppose $\jhat < j < j' < \jhat'.$
If $m \geq 3$, $h_{i\jhat} < h_{i\jhat'}$ and $h_{ij}<h_{ij'}$, then an $m$-coloring $(I,\phi,c)$ exists. 
The obtained summand $X_{(I,\phi,c)}$ in the $m$-fold coproduct (\ref{f33}) writes 
\begin{align*}
\langle a_{ij},a_{ij'}\rangle \langle a_{i\jhat},a_{i\jhat'}\rangle \hbar^{2} \;\cdot\;  &[(a_{i1},h_{i1}),\ldots,\widehat{(a_{i\jhat},h_{i\jhat})},(a_{i(\jhat'+1)},h_{i(\jhat'+1)}),\ldots,(a_{ip_{i}},h_{ip_{i}})]\otimes\\
&[(a_{i(\jhat+1)},h_{i(\jhat+1)}),\ldots,\widehat{(a_{ij},h_{ij})},(a_{i(j'+1)},h_{i(j'+1)}),\ldots,\widehat{(a_{i\jhat'},h_{i\jhat'})}]\otimes \\
&[(a_{i(j+1)},h_{i(j+1)}),\ldots,\widehat{(a_{ij'},h_{ij'})}],
\end{align*}
where $\widehat{(a_{i\jhat},h_{i\jhat})}$ indicates that $(a_{i\jhat},h_{i\jhat})$ is deleted from the component.
We remark that there exists no $2$-coloring for the components.
If $h_{i\jhat} > h_{i\jhat'}$, then the second term from the left in the above tensor product should be placed before the first.
As for the remaining cases, it suffices to consider the case $\jhat < j < \jhat' < j'$.
The obtained component 
\begin{align*}
&[(a_{i1},h_{i1}),\ldots,\widehat{(a_{i\jhat},h_{i\jhat})},(a_{i(\jhat'+1)},h_{i(\jhat'+1)}),\ldots,\\
&\;\widehat{(a_{ij'},h_{ij'})},(a_{i(j+1)},h_{i(j+1)}),\ldots,\widehat{(a_{i\jhat'},h_{i\jhat'})},(a_{i(\jhat+1)},h_{i(\jhat+1)}),\ldots,\\
&\;\widehat{(a_{ij},h_{ij})},(a_{i(j'+1)},h_{i(j'+1)}),\ldots,(a_{ip_{i}},h_{ip_{i}})].
\end{align*}
There exists no color function for this component because all the pairs map to the same color but $c(i,\jhat) \neq c(i,\jhat')$ is required.
Hence there is no such component in the $m$-fold coproduct (\ref{f33}).
Following the spirit of [Section~9]{Turaev}, we briefly draw a graph representing these components, where components are treated as ``knots'', and $\jhat,j$ stand for $(i,\jhat),(i,j)$ respectively.
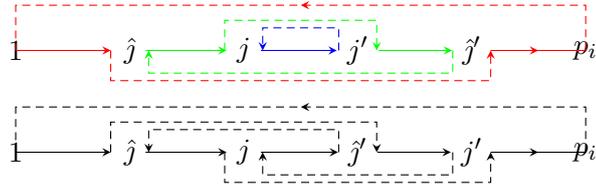
\begin{center}
\begin{tikzpicture}[>=stealth, node distance=1.5cm,decoration={markings,mark=at position 0.5 with {\arrow{>}}}]
    \node (1) at (-3.75,0) {1};
    \node (jhat1) at (-2.25,0) {$\jhat$};
    \node (j) at (-.75,0) {$j$};
    \node (jprime) at (.75,0) {$j'$};
    \node (jhat2) at (2.25,0) {$\jhat'$};
    \node (p1) at (3.75,0) {$p_i$};

    \draw[->, color=red] (-3.75,0) -- (-2.5,0);
    \draw[->, color=green] (jhat1) -- (j);
    \draw[->, color=blue] (j) -- (jprime);
    \draw[->, color=green] (jprime) -- (jhat2);
    \draw[-, color=red,postaction={decorate}] (2.5,0) -- (3.75,0);
    
    \draw[-, densely dashed, bend left=30, color=red,postaction={decorate}] (3.75,0) -| (3.75,.6) -| (-3.75,0); 
    \draw[->, densely dashed, bend right=30, color=red,] (-2.5,0) -| (-2.5,-.4) -| (2.5,0);
    \draw[->, densely dashed, bend left=30, color=green,] (2,0) -| (2,-.3) -| (-2,0);
    \draw[->, densely dashed, bend left=30, color=green,] (-1,0) -| (-1,.4) -| (1,0);
    \draw[->, densely dashed, bend right=30, color=blue,] (.5,0) -| (.5,.3) -| (-.5,0);
\end{tikzpicture}

\vspace{0.3cm}
\begin{tikzpicture}[>=stealth, node distance=1.5cm,decoration={markings,mark=at position 0.5 with {\arrow{>}}}]
    \node (1) at (-3.75,0) {1};
    \node (2) at (-2.25,0) {$\jhat$};
    \node (3) at (-.75,0) {$j$};
    \node (4) at (.75,0) {$\jhat'$};
    \node (5) at (2.25,0) {$j'$};
    \node (p1) at (3.75,0) {$p_i$};

    \draw[->] (-3.75,0) -- (-2.5,0);
    \draw[->] (2) -- (3);
    \draw[->] (3) -- (4);
    \draw[->] (4) -- (5);
    \draw[-,postaction={decorate}] (2.5,0) -- (3.75,0);
    
    \draw[-, densely dashed, bend left=30,postaction={decorate}] (3.75,0) -| (3.75,.6) -| (-3.75,0); 
    \draw[->, densely dashed, bend right=30] (-2.5,0) -| (-2.5,.4) -| (1,0);
    \draw[->, densely dashed, bend left=30] (.5,0) -| (.5,.3) -| (-2,0);
    \draw[->, densely dashed, bend left=30] (-1,0) -| (-1,-.4) -| (2.5,0);
    \draw[->, densely dashed, bend right=30] (2,0) -| (2,-.3) -| (-.5,0);
\end{tikzpicture}\end{center}

    \item 
(the case $i' = i$ and $\ihat' \neq i$) In this case, we have $c(i,\jhat) \neq c(i,\jhat')$ due to $h_{i\jhat} \neq h_{i\jhat'}$.
However in $X_{(I,\phi,c)}$, we have component $[\widehat{(a_{i\jhat},h_{i\jhat})},\ldots,\widehat{(a_{i\jhat'},h_{i\jhat'})},\ldots],$ which requires $c(i,\jhat) = c(i,\jhat'),$ a contradiction. 
    \item
(the case $\ihat' = i' \neq i$) 
Suppose $j<\jhat$ and $\jhat' < j'.$
If $h_{ij}<h_{i'j'}$ and $h_{i\jhat}>h_{i'\jhat'},$ then an $m$-coloring $(I,\phi,c)$ for $m\geq 2$ exists so that the pairs in the left (resp. right) component of the tensor maps to $1$ (resp. $2$).
The obtained summand $X_{(I,\phi,c)}$ in the $m$-fold coproduct (\ref{f33}) writes 
\begin{align*}\langle a_{ij},a_{ij'}\rangle \langle a_{i\jhat},a_{i\jhat'}\rangle h\hbar \;\cdot\;&[(a_{i1},h_{i1}),\ldots,\widehat{(a_{ij},h_{ij})},(a_{i'(j'+1)},h_{i'(j'+1)}),\ldots,\\
&\;(a_{i'p_{i'}},h_{i'p_{i'}}),(a_{i'1},h_{i'1}),\ldots,\\
&\;\widehat{(a_{i'\jhat'},h_{i'\jhat'})},(a_{i(\jhat+1)},h_{i(\jhat+1)}),\ldots,(a_{ip_{i}},h_{ip_{i}})]\otimes\\
&[(a_{i(j+1)},h_{i(j+1)}),\ldots,\widehat{(a_{i\jhat},h_{i\jhat})},(a_{i'(\jhat'+1)},h_{i'(\jhat'+1)}),\ldots,\widehat{(a_{i'j'},h_{i'j'})}].
\end{align*}
Next, suppose $j<\jhat$ and $j'<\jhat'$.
If $h_{ij}<h_{i'j'}$ and $h_{i\jhat}>h_{i'\jhat'},$ then an $m$-coloring $(I,\phi,c)$ for $m\geq 2$ exists.
The obtained summand $X_{(I,\phi,c)}$ in the $m$-fold coproduct (\ref{f33}) writes 
\begin{align*}\langle a_{ij},a_{ij'}\rangle \langle a_{i\jhat},a_{i\jhat'}\rangle h\hbar \;\cdot\;&[(a_{i1},h_{i1}),\ldots,\widehat{(a_{ij},h_{ij})},(a_{i'(j'+1)},h_{i'(j'+1)}),\ldots,\\
&\;\widehat{(a_{i'\jhat'},h_{i'\jhat'})},(a_{i(\jhat+1)},h_{i(\jhat+1)}),\ldots,(a_{ip_{i}},h_{ip_{i}})]\otimes\\
&[(a_{i(j+1)},h_{i(j+1)}),\ldots,\widehat{(a_{i\jhat},h_{i\jhat})},(a_{i'(\jhat'+1)},h_{i'(\jhat'+1)}),\ldots,\\
&\;(a_{i'p_{i'}},h_{i'p_{i'}}),(a_{i'1},h_{i'1}),\ldots,\widehat{(a_{i'j'},h_{i'j'})}].
\end{align*}
As in the case (1), we also draw a graph representing these components.
\begin{center}
\begin{tikzpicture}[>=stealth, node distance=1.5cm,decoration={markings,mark=at position 0.5 with {\arrow{>}}}]
    \node (1) at (-5.25,0) {1};
    \node (2) at (-3.75,0) {$j$};
    \node (3) at (-2.25,0) {$\jhat$};
    \node (pi) at (-.75,0) {$p_{i}$};
    \node (1p) at (.75,0) {$1$};
    \node (2p) at (2.25,0) {$\jhat'$};
    \node (3p) at (3.75,0) {$j'$};
    \node (pip) at (5.25,0) {$p_{i'}$};
    \node at (0,0) {$\otimes$};
    
    \draw[-, color=red,postaction={decorate}] (-5.25,0) -- (-4,0);
    \draw[->, color=blue] (2) -- (3);
    \draw[-, color=red,postaction={decorate}] (-2,0) -- (-.75,0);
    \draw[-, color=red,postaction={decorate}] (.75,0) -- (2,0);
    \draw[->, color=blue] (2p) -- (3p);
    \draw[-, color=red,postaction={decorate}] (4,0) -- (5.25,0);
    
    \draw[-, densely dashed, bend left=30, color=red, postaction={decorate}] (-1,0) -| (-0.75,.5) -| (-5.25,0); 
    \draw[->, densely dashed, bend left=30, color=red] (-4,0) -| (-4,-.7) -| (4,0); 
    \draw[-, densely dashed, bend left=30, color=red, postaction={decorate}] (5.25,0) -| (5.25,.5) -| (0.75,0);
    \draw[->, densely dashed, bend left=30, color=red] (2,0) -| (2,-.3) -| (-2,0);
    \draw[->, densely dashed, bend right=30, color=blue] (3.5,0) -| (3.5,-.6) -| (-3.5,0);
    \draw[->, densely dashed, bend right=30, color=blue] (-2.5,0) -| (-2.5,-.4) -| (2.5,0);
\end{tikzpicture}

\vspace{0.3cm}
\begin{tikzpicture}[>=stealth, node distance=1.5cm,decoration={markings,mark=at position 0.5 with {\arrow{>}}}]
    \node (1) at (-5.25,0) {1};
    \node (2) at (-3.75,0) {$j$};
    \node (3) at (-2.25,0) {$\jhat$};
    \node (pi) at (-.75,0) {$p_{i}$};
    \node (1p) at (.75,0) {$1$};
    \node (2p) at (2.25,0) {$j'$};
    \node (3p) at (3.75,0) {$\jhat'$};
    \node (pip) at (5.25,0) {$p_{i'}$};
    \node at (0,0) {$\otimes$};
    
    \draw[-, color=red, postaction={decorate}] (-5.25,0) -- (-4,0);
    \draw[->, color=blue] (2) -- (3);
    \draw[-, color=red, postaction={decorate}] (-2,0) -- (-.75,0);
    \draw[-, color=blue, postaction={decorate}] (.75,0) -- (2,0);
    \draw[->, color=red] (2p) -- (3p);
    \draw[-, color=blue, postaction={decorate}] (4,0) -- (5.25,0);
    
    \draw[-, densely dashed, bend left=30, color=red, postaction={decorate}] (-.75,0) -| (-.75,.7) -| (-5.25,0);
    \draw[->, densely dashed, bend left=30, color=blue, postaction={decorate}] (5.25,0) -| (5.25,.7) -| (.75,0);
    \draw[->, densely dashed, bend left=30, color=red] (-4,0) -| (-4,.4) -| (2.5,0);
    \draw[->, densely dashed, bend left=30, color=red] (3.5,0) -| (3.5,-.3) -| (-2,0);
    \draw[->, densely dashed, bend left=30, color=red] (3.5,0) -| (3.5,-.3) -| (-2,0);
    \draw[->, densely dashed, bend left=30, color=blue] (2,0) -| (2,.3) -| (-3.5,0);
    \draw[-, densely dashed, bend left=30, color=blue, postaction={decorate}] (-2.5,0) -| (-2.5,-.4) -| (4,0);
    
\end{tikzpicture}\end{center}
\end{enumerate}
\end{example}

\end{document}